\documentclass[11pt]{amsart} 
\usepackage[english]{babel}
\usepackage{graphicx,epsfig}
\usepackage{bbm}
\usepackage{cite}
\usepackage{amsthm}
\usepackage{esvect}
\usepackage{amsmath,amssymb,latexsym, amsfonts, amscd, amsthm, xy}
\input{xy}
\xyoption{all}

\usepackage{mathrsfs}

\usepackage{mathtools}

\usepackage[margin=1in]{geometry}

\usepackage[dvipsnames]{xcolor}

\makeindex 

\usepackage{hyperref}
\hypersetup{pdftoolbar=true, pdftitle={GKform}, pdffitwindow=true, colorlinks=true, citecolor=blue, filecolor=black, linkcolor=purple, urlcolor=red, hypertexnames=false}

\theoremstyle{plain}

\newtheorem{theorem}{Theorem}[section]

\newtheorem{proposition}[theorem]{Proposition}
\newtheorem{corollary}[theorem]{Corollary}

\newtheorem{assumption}[theorem]{Assumption}
\newtheorem*{assumption*}{Assumption}
\newtheorem{lemma}[theorem]{Lemma}
\newtheorem{question}[theorem]{Question}

\theoremstyle{definition}
\newtheorem{definition}[theorem]{Definition}
\newtheorem{remark}[theorem]{Remark}

\newtheorem*{goal*}{Goal}
\newtheorem*{problem*}{Comment}

\newtheorem{notation}[theorem]{Notation}

\def\lra{{\; \longrightarrow \;}}

\def\G{{\bf G}}

\def\PP{{\mathbf P}}

\DeclareMathOperator{\NS}{NS}
\DeclareMathOperator{\Gm}{\mathbb{G}_m}
\DeclareMathOperator{\sm}{sm}

\DeclareMathOperator{\tors}{tors}

\DeclareMathOperator{\codim}{codim}

\DeclareMathOperator{\Res}{Res}

\DeclareMathOperator{\Pic}{Pic}

 \DeclareMathOperator{\pr}{pr}
\DeclareMathOperator{\id}{id}

\DeclareMathOperator{\Bl}{Bl}
 \DeclareMathOperator{\cl}{cl}

 \DeclareMathOperator{\Div}{Div}
\DeclareMathOperator{\Hom}{Hom}
\DeclareMathOperator{\Hombf}{{\bf Hom}}

\DeclareMathOperator{\spec}{Spec}

\DeclareMathOperator{\tf}{\mathrm{tf}}

\DeclareMathOperator{\rank}{rank}

\DeclareMathOperator{\tr}{t}

\DeclareMathOperator{\im}{Im}

\DeclareMathOperator{\Jac}{Jac}

\def\fp{\mathfrak{p}}

\def\d{\mathrm{d}}

\def\TT{\mathbf{T}}
\def\Z{\mathbb{Z}}

\def\F{\mathbb{F}}

\def\BB{\mathbf{B}}
\def\AA{\mathbf{A}}

\def\Q{\mathbb{Q}}

\def\J{\mathbf{J}}

\def\Jo{\mathbf{J}^{\vee,\circ}}

\def\G{\mathbb{G}}

\def\P{\mathbf{P}}
\def\CC{\mathbf{C}}

\def\bdf{\begin{defn}}
\def\edf{\end{defn}}

\def\ra{\rightarrow}

\def\d1{d^{(1)}}

\def\Gm{{\mathbb{G}_m}}

\def\d{\mathbf{d}}

\def\U{\mathbf{U}}
\def\Y{\mathbf{Y}}

\def\oh{\mathcal{O}}
\def\longhookrightarrow{\lhook\joinrel\longrightarrow}

\usepackage{tikz}
\usetikzlibrary{positioning} 
\usepackage{tikz-cd}
\usetikzlibrary{arrows,graphs,decorations.pathmorphing,decorations.markings}

\tikzset{
commutative diagrams/.cd,
arrow style=tikz,
diagrams={>=latex}}

\numberwithin{equation}{section}

\usepackage{microtype}
\begin{document} 

 \title{Geometric Quadratic Chabauty over Number Fields}
\author{Pavel \v{C}oupek, David T.-B. G. Lilienfeldt, Luciena X. Xiao, Zijian Yao}
\address[Pavel \v{C}oupek]{Department of Mathematics, Purdue University}
\email{pcoupek@purdue.edu}
 \address[David T.-B. G. Lilienfeldt]{Einstein Institute of Mathematics, Hebrew University of Jerusalem}
\email{davidterborchgram.lilienfeldt@mail.huji.ac.il}
 \address[Luciena X. Xiao]{Department of Mathematics and Statistics, University of Helsinki}
\email{xiao.xiao@helsinki.fi}
\address[Zijian Yao]{Department of Mathematics, Universit\'e Paris Saclay/CNRS}
\email{zijian.yao.math@gmail.com}
\subjclass[2020]{11G30, 11D45, 14G05}
\keywords{Rational points, geometric quadratic Chabauty, Poincar\'e torsor, biextension}

\dedicatory{Dedicated to the memory of Bas Edixhoven}

\begin{abstract}
This article generalizes the geometric quadratic Chabauty method, initiated over $\Q$ by Edixhoven and Lido, to curves defined over arbitrary number fields. The main result is a conditional bound on the number of rational points on curves that satisfy an additional Chabauty type condition on the Mordell--Weil rank of the Jacobian.
The method gives a more direct approach to the generalization by Dogra of the quadratic Chabauty method to arbitrary number fields. 
\end{abstract}

\maketitle

\tableofcontents

\section{Introduction}

\subsection{Main result}

Let $C$ be a smooth, projective, geometrically irreducible curve of genus $g \ge 2$ defined over a number field $K$. Let $J:= \Pic_{C/K}^0$ denote the Jacobian of $C$. Let 
\begin{itemize}
\item $r:=\rank_\Z J(K)$ be the Mordell--Weil rank, and 
\item $\rho := \textup{rank}_{\Z} \NS(J)$ be the rank of the N\'eron--Severi group of $J$. 
\end{itemize}
  Mordell's conjecture  (now a theorem of Faltings \cite{faltings}) asserts that the set of rational points on $C$ is finite. Faltings' spectacular proof, however, cannot be made effective and there is no general algorithm for determining the set $C(K)$\footnote{though recently Alp\"oge and Lawrence found an algorithm  that terminates assuming standard conjectures (see \cite[Ch. 7-9]{alpogethesis}).}. 
More recently, following the pioneering work of Chabauty \cite{chab41}, methods have been developed to explicitly determine the set of rational points on curves that satisfy certain conditions on $r$ (commonly referred to as Chabauty type conditions).  Some relevant results in this direction will be surveyed in the introduction. Let us fix a prime $p$ where $C$ has good reduction while satisfying some additional mild ramification conditions (see Assumption \ref{passumption}). In its crudest form, our main result is an effective bound on $C(K)$ of the following form, which generalizes the work of Edixhoven and Lido \cite{EL19}.

\begin{theorem}[crude version]\label{thm:intro_main}
Suppose that $C$ satisfies the ``quadratic Chabauty condition''
\begin{equation}\label{eq:geom_cc}
r+\delta(\rho-1) \leq (g+\rho-2)d,
\end{equation}
where $\delta=\rank_\Z \oh_K^\times$ and $d$ is the degree of $K$. Let $R:=\Z_p \langle z_1, ..., z_{r+\delta (\rho - 1)} \rangle$
 be the $p$-adically completed polynomial algebra over $\Z_p$. There exists an ideal $I$ of $R$, which is explicitly computable mod $p$, such that if $\overline A:= (R/I) \otimes_{\Z_p} \F_p$  is finite dimensional over $\F_p$, then the set of rational points $C(K)$ is finite, and its cardinality is bounded by $ \dim_{\F_p} \overline A.$
 \end{theorem}
 
  \begin{remark} \label{remark:main_intro}
The precise form of this theorem (Corollary \ref{cor:GQCNF}) is slightly more involved. As in the work of Edixhoven and Lido, we need to work integrally with a regular proper model $\CC$ of $C$ over $\oh_K$ and cover the smooth locus $\CC^{\sm}$ by certain open subschemes $\U_i$. More precisely, for each $p$-adic residue disk of $\U_i$ (indexed by the residue points $u\in \U_i(\oh_K\otimes \F_p)$), we produce a bound on the size of $\U_{i}(\oh_K)_u$ by constructing an ideal $I_{i, u} \subset R$. The bound on the size of $C(K)$ is then obtained by summing the bounds for each $i$ and $u$.  
 \end{remark}


\subsection{Effective methods over $\Q$}\label{s:Qmethods}

To put our result into context, we give a brief overview of some of the existing effective methods. 
 
\subsubsection{The Chabauty--Coleman method}\label{s:intro:CC}
If the Mordell--Weil rank $r$ of the Jacobian $J$ of $C$ satisfies the inequality  
$r < g$, then the work of Chabauty \cite{chab41} and Coleman \cite{chabcol}  can be used to give upper bounds on the size of $C(K)$, and in many cases, to explicitly compute the set of rational points. Their method works over general number fields, but let us take $K = \Q$ to explain their strategy.
Upon choosing a sufficiently large prime $p$ of good reduction, one obtains a homomorphism 
$$\log_p: J (\Q_p) \lra H^0 (C_{\Q_p}, \Omega^1)^{\vee} \simeq H^0 (J_{\Q_p}, \Omega^1)^{\vee} $$
using Coleman integration (see \cite{coleman} for details). 
The Abel--Jacobi map $j_b: C \hookrightarrow J$, which relies on a fixed base point $b \in C(\Q)$, leads to the following commutative diagram:
\begin{equation}\label{diag:chabcol}
\begin{tikzcd}
C(\Q) \arrow{r} \arrow{d}{j_b} & C(\Q_p)  \arrow{d}{j_b} \arrow{rd}{\int} \\
J(\Q) \arrow{r} & J (\Q_p)   \arrow{r}[swap]{\log_p} & H^0 (C_{\Q_p}, \Omega^1)^{\vee}.
\end{tikzcd}
\end{equation}
The Chabauty condition $r < g$ guarantees that the closure $\overline{J(\Q)}^p$ of $J(\Q)$ in $J (\Q_p)$ with respect to the $p$-adic topology has positive codimension. 
 In particular, there exists a nontrivial differential form $\omega$ for which $\log_p(\omega)$ vanishes on $\overline{J(\Q)}^p$. Roughly, on each residue disk of the curve $C$, the Coleman integral $\log_p(\omega)\circ j_b=\int \omega$ is given by a convergent $p$-adic power series and has only finitely many zeros. Since it vanishes on $C(\Q)$, this set is finite.  
Moreover, using Newton polygons, Coleman \cite{chabcol} was able to count the number of zeros of these $p$-adic power series and prove, when $p>2g$, that
\[
\vert C(\Q) \vert \leq \vert C(\F_p) \vert+(2g-2). 
\]

\subsubsection{Quadratic Chabauty}\label{s:intro:QC}
The tantalizing non-abelian Chabauty program, initiated by Kim \cite{Kim05, Kim09}, aims to relax the Chabauty condition $r<g$ by considering non-abelian variants of the objects in Section \ref{s:intro:CC}. To this end one first reinterprets the diagram \eqref{diag:chabcol} above using the Bloch--Kato Selmer groups $H^1_f(\Q, V)$ (resp. $H^1_{f} ( {\Q_p}, V)$) in place of $J(\Q)$ (resp. $J(\Q_p)$) via the Kummer maps, where $V:=V_p J$ denotes the $p$-adic Tate module of $J$. The logarithm map above is essentially the inverse of the Bloch--Kato exponential 
$$ H^0 (C_{\Q_p}, \Omega^1)^\vee \simeq \textup{D}_{\textup{dR}} (V)/\textup{D}^+_{\textup{dR}} (V) \xrightarrow{\:\: \text{ exp } \:\:} H^1_f (\Q_p, V).$$
Next one replaces $V$ by certain pro-unipotent quotients $U_n$ of the \'etale fundamental group $\pi_1^{\text{\'et}} (C_{\overline \Q})_{\Q_p}$, one for each $n \ge 1$, which again carry a continuous Galois action. Kim then defines a certain Selmer subgroup $\textup{Sel}(U_n) \subset H^1_f (\Q, U_n)$, and upgrades the previous diagram to the following one:
\[
\begin{tikzcd}
C(\Q) \arrow[r] \arrow[d, "j_n"] & C(\Q_p)  \arrow[d, "j_{n, p}"] \arrow[rd, "\int"] \\
\textup{Sel}(U_n) \arrow{r}[swap]{\textup{loc}_p} & H^1_{f} (\Q_p, U_n)  \arrow{r}[swap]{\textup{loc}_n}& \pi_1^{\textup{dR}} (C_{\Q_p})_{n}/ \textup{Fil}^0.
\end{tikzcd}
\]
Here the vertical maps $j_n$ and $j_{n, p}$ are Kim's unipotent Kummer maps. Define the (nested) sets
$$ C(\Q_p)_n : = j_{n, p}^{-1} \big(\textup{loc}_p (\textup{Sel}(U_n))\big) \subset C(\Q_p),  $$
which contain $C(\Q)$ and satisfy $C(\Q_p)_{n+1} \subset{C(\Q_p)_n}$. 
For sufficiently large $n$, Kim conjectures that $C(\Q_p)_n$ is finite and even coincides with $C(\Q)$. The set $C(\Q_p)_1$ is the one studied in the Chabauty--Coleman method in Section \ref{s:intro:CC}. 
The first non-abelian instance of Kim's program (namely the quadratic case where $n = 2$) has been carried out quite successfully by Balakrishnan,   Dogra, and their collaborators (see \cite{BD18, BD19, BD20} and the references therein). In particular, they show that  if the Mordell--Weil rank $r$ satisfies $r < g + \rho - 1$,  then $C(\Q_p)_2$ is finite and can often be explicitly determined.  This method has been applied to determine the rational points of the ``cursed curve'' in \cite{balakrishnan2019} by Balakrishnan, Dogra, M\"uller, Tuitman, and Vonk, which completes the classification of non-CM elliptic curves over $\Q$ of split Cartan type.

\subsection{Geometric quadratic Chabauty (over $\Q$)}\label{s:intro:GQC}
Recently, Edixhoven and Lido have explored in \cite{EL19} a different, and arguably more direct, approach to quadratic Chabauty. Their method is referred to as the geometric quadratic Chabauty method (with an emphasis on geometric), and works under the same  condition $r < g+\rho - 1$ as in Section \ref{s:intro:QC}. It has the advantage of avoiding the consideration of iterated Coleman integrals and the analysis of certain complicated $p$-adic heights. 
The strategy of Edixhoven and Lido is close in spirit to the original idea of Chabauty: their idea is to replace the Jacobian $J$ by a certain $\G_m^{\rho-1}$-torsor $T$ over $J$ in order to relax the condition $r<g$. This torsor comes equipped, by construction, with a lift $\tilde{j}_b : C\lra T$ of the Abel--Jacobi map. Upon choosing a prime $p$ of good reduction, one obtains the following diagram
\begin{equation}\label{diag:EL}
 \begin{tikzcd}  C(\Q)  \arrow{d}{\tilde{j}_b}  \arrow[rr] &&  C(\Q_p) \arrow{d}{\tilde{j}_b} \\   T(\Q)  \arrow[r] & \overline{T(\Q)}^p \arrow[r]   & T(\Q_p), 
 \end{tikzcd} 
 \end{equation}
 where $\overline{T(\Q)}^p$ denotes the closure of $T(\Q)$ in $T(\Q_p)$ with respect to the $p$-adic topology. In practice, they need to work with certain integral versions of the above objects and diagrams over $\spec(\Z)$ (see Remark \ref{remark:integral_model}). The method then consists in bounding the size of the intersection 
 \begin{equation}\label{interEL}
 \tilde{j}_b^{-1}(\overline{T(\Z)}^p\cap \tilde{j}_b(C(\Z_p))),
 \end{equation}
 which contains the set of rational points. Using this method, Edixhoven and Lido are able to reprove Faltings' theorem for curves defined over $\Q$ satisfying $r < g+\rho - 1$. Furthermore, they have made their method effective and successfully used it to compute the rational points on the quotient of the modular curve $X_0(129)$ by the Atkin--Lehner group $\langle w_3, w_{43}\rangle$. 

\begin{remark} \label{remark:integral_model}   The reason for working with integral models is that the torsor $T$ has too many rational points (its fiber over $J$ is $\G_m^{\rho - 1}$). More precisely, one starts with a regular proper integral model $\CC$ of $C$, and further restricts to open subschemes on which the Abel--Jacobi map can be lifted integrally (as mentioned in Remark \ref{remark:main_intro}). 
\end{remark}

\subsection{Effective methods over number fields}\label{s:intro:GQCNF}

The methods summarized so far deal with curves defined over $\Q$ satisfying Chabauty type conditions, for example, 
\[
\begin{cases}
r<g & \qquad \qquad \text{ (Chabauty condition)} \\
r<g+\rho-1 & \qquad \qquad \text{ (quadratic Chabauty condition)}.
\end{cases}
\]
Next we briefly review some recent developments in generalizing the methods in Section \ref{s:Qmethods}  to curves over arbitrary number fields.

\subsubsection{The Chabauty--Coleman method}\label{s:intro:CCK}

As remarked earlier, the Chabauty--Coleman method works over an  arbitrary number field $K$. In fact, Coleman already works at this level of generality in  \cite{chabcol}. 
 

\subsubsection{The Restriction of Scalars (RoS) Chabauty method}\label{s:intro:RoSC}

In \cite{siksek}, Siksek extends the Chabauty--Coleman method by studying the Weil restrictions from $K$ to $\Q$ of the curve $C$ and its Jacobian $J$. In this way Siksek reduces the problem to working entirely over $\Q$ at the cost of considering higher dimensional varieties. 
Siksek's method, known as Restriction of Scalars (RoS) Chabauty, requires the RoS Chabauty condition $r\leq (g-1)d,$ where $d = [K:\Q]$ is the degree of $K$. Note, however, that this method can fail to produce a bound on the number of rational points even if the condition $r \le (g-1)d$ is satisfied. Examples include the case where the curve $C$ is the base change of a curve $C'$ defined over $\Q$, which does not satisfy the Chabauty condition $\rank_\Z \Jac(C')<g$. Aware of this, Siksek asks in \cite{siksek} whether a sufficient condition for his method to prove finiteness is that for all extensions $\Q\subset L\subset K$ over which $C$ admits a good model $C_L$, we should have
\[
\rank_\Z \Jac(C_L)\leq (g-1)[L:\Q].
\]
Failures of the method of RoS Chabauty have been studied by Triantafillou \cite{triantafillou}, who introduces Base-Change-Prym (BCP) obstructions, which account for all known failures to date.  

\subsubsection{RoS quadratic Chabauty}\label{s:intro:RoSQC}

More recently, Dogra \cite{Dogra19} combines ideas of the RoS Chabauty method with Kim's non-abelian Chabauty program. This has led to a generalization of Kim's program to arbitrary number fields. As in Section \ref{s:intro:QC}, he obtains (nested) Chabauty--Kim 
sets $C(K_\fp)_n$ for $\fp \mid p$, indexed by $n \in \Z_{\ge 1}$. 
Dogra provides a negative answer to Siksek's question using a BCP obstruction.  Assuming $p$ is a split prime of good reduction, he also gives a sufficient condition for there to exist a prime $\fp\mid p$ such that $C(K_\fp)_1$ is finite under the Chabauty condition $r\leq (g-1)d$: namely that 
\begin{equation}\label{hom}
\Hom(J_{\bar{\Q}, \sigma_1}, J_{\bar{\Q}, \sigma_2})=0  \text{ for any two distinct embeddings } \sigma_1, \sigma_2 : K\hookrightarrow \bar{\Q}.
\end{equation}
Moreover, his results imply that if in addition to Condition \eqref{hom} the quadratic (RoS) Chabauty condition (\ref{eq:geom_cc})
as in our main theorem is satisfied, then there exists  $\fp \mid p$ such that $C(K_\fp)_2$ is finite. The work of Dogra sets the theoretical stage for the quadratic RoS Chabauty method, which has been made effective recently by Balakrishnan, Besser, Bianchi, and M\"uller in \cite{BBBM19} for odd degree hyperelliptic curves and genus $2$ bielliptic curves. 

\subsection{Overview of the geometric method}\label{s:overview}
Next we explain the proof of Theorem \ref{thm:intro_main}. Following suggestions by the referee, we first describe the geometric method in a more general setup. Our main theorem can then be viewed as carrying out this method in a special case. To this end, let $X_K$ be a proper variety over a number field $K$ which admits a regular proper model $\mathbf{X}$ over $\mathcal{O}_K$. The key idea from \cite{EL19} is to replace the Jacobian in Chabauty's original approach by something higher dimensional. To phrase their approach in a more general setup, let us \textit{assume} that we are given closed embeddings
$$ X_K \xrightarrow{ \: j \: } J_K \xrightarrow{ \: f \: } A_K \times B_K $$
where $J_K, A_K, B_K$ are abelian varieties and $f$ is a homomorphism. Let us further \textit{assume} that there is a $G_K$-biextension $F_K$ of $(A_K, B_K)$ where $G_K$ is a torus, such that the map $j$ lifts to an embedding $\tilde j:  X_K \ra Q_K$, where $Q_K$ is the $G_K$-torsor over $J_K$ obtained as the pullback of $F_K$ along $f$:  
\begin{equation}\label{diag:AJliftQK}
\begin{tikzcd}[bend angle = 20]
  & Q_K  \arrow{r} \arrow{d} \arrow[phantom]{dr}{\square}
  & F_K \arrow{d} \\
  X_K \arrow{r}{j} \arrow[rightarrow]{ru}{\tilde{j}}
  & J_K \arrow{r}{f}
  & A_K\times B_K.
\end{tikzcd}
\end{equation} 
In view of Remark \ref{remark:integral_model}, we \textit{assume} that the diagram above spreads out over $\mathcal{O}_K$:
\begin{equation}\label{diag:AJliftUQ1}
\begin{tikzcd}[bend angle = 20]
  & \mathbf{Q}  \arrow{r} \arrow{d} \arrow[phantom]{dr}{\square}
  & \mathbf{F} \arrow{d} \\
  \mathbf{X}   \arrow{r}{j}    \arrow[rightarrow]{ru}{\tilde{j}}
  & \J \arrow{r}{f}
  & \AA\times \BB^\circ.
\end{tikzcd}
\end{equation}
Here the abelian varieties $A_K, J_K$ (resp. $B_K$) are replaced by their N\'eron models (resp. the connected component of identity of its N\'eron model), and $\mathbf F$ is the (unique) $\mathbf G$-biextension of $(\mathbf A, \mathbf B^{\circ})$ where $\mathbf G$ is the N\'eron--Raynaud model of $G_K$. 
\begin{remark} 
In practice, we often have to replace $\mathbf X$ by a Zariski open cover $\mathbf U_i$ of $\mathbf{X}^{\sm}$ (where $\mathbf X^{\sm}$ denotes the smooth locus in $\mathbf{X}$),  and work with one $\mathbf U_i$ at a time. 
\end{remark}
 
 \noindent For a prime $p$ of good reduction for $\mathbf{X}$, we let $\oh_{K, p} := \oh_K \otimes_{\Z} \Z_p$ and $\overline{\oh_{K,p}}:=(\oh_K \otimes \mathbb{F}_p)_{\mathrm{red}}$. Working residue disk by residue disk, we fix $u\in \mathbf{X}(\overline{\oh_{K,p}})$ and $t=\tilde{j}(u)\in \mathbf{Q}(\overline{\oh_{K,p}})$. Consider the following commutative quadratic Chabauty diagram on residue disks:
\begin{equation} \label{diagram:introX}
 \begin{tikzcd}  \mathbf X(\oh_{K})_u  \arrow[d,  "\tilde j"]  \arrow[rr] &&  \mathbf X (\oh_{K, p})_u \arrow[d, "\tilde j"] \\   \mathbf Q (\oh_K)_t  \arrow[r] & \Y_t \arrow[r]   & \mathbf Q (\oh_{K, p})_t. \end{tikzcd} 
 \end{equation}
Here $\Y_t := \overline{\mathbf Q (\oh_K)_t}^p$ is the closure of $\mathbf Q (\oh_K)_t$ in $\mathbf Q(\oh_{K, p})_t$ for the $p$-adic topology. 
The points in $X_K (K) = \mathbf X (\oh_K)$ that reduce to $u$ modulo $p$ are contained in the intersection  $ \tilde{j}(\mathbf X(\oh_{K, p}))_u  \cap \Y_t,$ which is often computable. The framework to explicitly determine this intersection can be described as follows: using the structure of the rational points on $J_K$ and the $\mathbf G$-torsor structure of $\mathbf Q \ra \mathbf J$, one constructs a certain ``coordinate map'' $\Z^{m} \ra \mathbf Q(\mathcal O_K)_t$, from which one can build a map $\kappa: \Z_p^m \twoheadrightarrow \mathbf Y_t$ between $p$-adic analytic spaces, where $m=\rank_\Z J_K (K) + \rank_\Z \mathbf G (\oh_K)$. The map $\kappa$ is then used to give effective bounds on the intersection   $ \tilde{j}(\mathbf X(\oh_{K, p})_u)  \cap \Y_t$ when it is finite. In the next subsection, we explain this procedure in more detail for the primary example treated in this article, namely when $X_K$ is a higher genus curve over $K$. 

\begin{remark} 
While in principle we expect the method of this article to work whenever we are given the setup as in diagram (\ref{diag:AJliftUQ1}), in practice is seems rare to find such a lift $\tilde j$. This is the main reason why we choose to treat only the case when $X_K$ is a higher genus curve. For example, in \cite{caropasten} a Chabauty--Coleman type bound was proved for hyperbolic surfaces contained in abelian varieties. However, for a hyperbolic surface contained in a principally polarized abelian variety, we do not know how to construct a torsor which fits in (the analogue of) Diagram (\ref{diag:AJliftUQ1}).
\end{remark}

\subsection{Higher genus curves over number fields}\label{s:HGC}
Let $C_K$ be a curve of genus $\ge 2$ over $K$ and consider a regular proper model $\CC$ of $C_K$ over  $\oh_K$. In this case, we will produce Diagram (\ref{diag:AJliftUQ1}) by considering the Poincar\'e bundle over $J_K \times J_K^\vee$ where $J_K$ denotes the Jacobian of $C_K$. In order to produce the desired lift $\tilde j_b$ (dependent on a base point $b\in C_K(K)$), we in fact consider the $(\rho-1)$-fold self-product of the Poincar\'e bundle over $J_K$, where $\rho$ is the N\'eron--Severi rank. Our map $J_K \ra J_K \times J_K^{\vee}$ in diagram (\ref{diag:AJliftQK}) depends on a choice of (sub-)basis of the N\'eron--Severi group, and is obtained by translating these basis elements by certain points of $J_K^\vee(K)$ that are closely related to Chow--Heegner points (see Section \ref{trivialization} and Remark \ref{remark:Chow_Heegner}). With this carefully chosen map, we obtain a $\G_m^{\rho-1}$-torsor $T_K$ over $J_K$, which spreads out to a $\G_m^{\rho-1}$-torsor $\TT$ over $\J$, the latter being the N\'eron model of $J_K$. 
The construction of $\TT$ allows us to lift the Abel--Jacobi map $j_b: \CC^{\sm} \lra \J$ to a map\footnote{ upon replacing $\CC^{\sm}$ by a certain Zariski open cover}
$$ \tilde j_b : \CC^{\sm}  \lra  \TT \:\: $$ 
as required by Diagram (\ref{diag:AJliftUQ1}). 
Fixing $u\in \CC^{\sm}(\overline{\oh_{K,p}})$ and letting $t:=\tilde j_b(u)\in \TT(\overline{\oh_{K,p}})$, we thus arrive at the following commutative diagram on residue disks:
\begin{equation} \label{diagram:intro}
 \begin{tikzcd}  \CC^{\sm}(\oh_{K})_u  \arrow[d,  "\tilde j_b"]  \arrow[rr] &&  \CC^{\sm}(\oh_{K, p})_u \arrow[d, "\tilde j_b"] \\   \TT(\oh_K)_t  \arrow[r] & \Y_t \arrow[r]   & \TT(\oh_{K, p})_t. \end{tikzcd} 
 \end{equation} 
The key of the approach is to analyze the $p$-adic closure $\Y_t$ of the $\oh_K$-points of the torsor $\TT$ reducing to $t$ modulo $p$. If $K = \Q$, this can be done by parametrizing the $p$-adic closure of $\J(\Z) = J (\Q)$, as $\Gm (\Z) = \{ \pm 1\}$. This is a major simplification and essentially why \cite{EL19} decided to work over $\Q$. In fact, it was suggested to us by Edixhoven and Lido that a restriction of scalars approach might reduce the case of general number fields $K$ back to the case of $\Q$. In this work, however, we decide to take a more direct approach, which departs from the RoS arguments of \cite{siksek, BBBM19, Dogra19}  (see Remark \ref{rem:singleprime} below). One of the main observations is that one can fully utilize the $\Gm$-action on the fibers of the torsor $\TT \lra \J$ to parametrize $\Y_t$, which is sufficient for the purpose of this paper. Roughly, we pick a ``$\Z$-coordinate'' map $\Z^{r} \lra \TT(\oh_K)_t$, essentially by choosing a basis for the Mordell--Weil group $J_K(K)$. We then use the $\Gm$-action to propagate these coordinates to get a ``$\Z$-coordinate'' map $\Z^{\delta(\rho - 1) + r} \lra \TT(\oh_K)_t$. 
Finally, interpolating these coordinates $p$-adically allows us to parametrize $\Y_t$ via a surjective map $$\kappa: \Z_p^{\delta (\rho - 1)+r} \twoheadrightarrow \Y_t, $$ 
which turns out to be given by convergent $p$-adic power series. The ideal $I$ in Theorem \ref{thm:intro_main} is built such that the cardinality of  $\spec (R/I)(\Z_p)$ is the size of $\kappa^{-1} (\Y_t \cap \tilde{j}_b(\CC^{\sm}(\oh_{K, p})_u))$. In particular, in order for the method to effectively determine the rational points on $C_K$, we need to choose a prime $p$ such that  
\begin{itemize} 
\item $\Y_t \cap \tilde{j}_b(\CC^{\sm} (\oh_{K, p})_u)$ is finite, and
\item $\kappa$ is finite-to-one on $ \kappa^{-1} (\Y_t \cap \tilde{j}_b(\CC^{\sm}(\oh_{K, p})_u))$. 
\end{itemize}
This observation prompts the following question (see Section \ref{s:question}): 

\begin{question}\label{q1}
Let $p$ be a prime of good reduction for $\CC$. What conditions would guarantee that the intersection $\Y_t \cap \tilde{j}_b(\CC^{\sm} (\oh_{K, p})_u)$ is finite? 
\end{question}

\begin{remark} 
There are some further caveats hidden from the overview above. For example, as hinted at in Remark \ref{remark:main_intro}, in order to obtain the lift $\widetilde j_b: \CC^{\sm} \ra \TT$, we need to restrict to certain open subschemes $\U_i$ of $\CC^{\sm}$. The details of these complications are spread throughout the article (see for instance Sections \ref{trivialization} and \ref{s:construct}). 
\end{remark}

\begin{remark}
As our approach has a theoretical component (the geometric method) and an effective component, it is appropriate to distinguish between the required Chabauty type conditions. The requirement for the geometric method is the condition \eqref{eq:geom_cc}, called the geometric quadratic Chabauty condition. Effectiveness relies on the map $\kappa$ being finite-to-one, for which we need to (at least) impose the additional condition 
$
\dim \Y_t = r+\delta(\rho-1).
$
The combination of the two conditions is the ``effective geometric quadratic Chabauty condition''. 
We refer to Section \ref{s:question} for further details.
\end{remark}


\begin{remark}[Restriction of scalars]\label{rem:singleprime}
In our approach we work with all primes of $K$ above $p$ simultaneously (similar to the work of Siksek). If we were to work with a single fixed prime over $p$, the method would only have a chance of working if the following condition was satisfied: 
\[
r+\delta(\rho-1) < g+\rho-1.
\]
When $K$ is imaginary quadratic, this amounts to the same quadratic Chabauty condition as over $\Q$. However, if $K$ is real quadratic, the condition becomes $r<g$ and the Chabauty--Coleman method can already be applied. When considering higher degree number fields, the above condition is more restrictive than the classical Chabauty condition. 
Therefore, it seems necessary to work with all primes above $p$ simultaneously in order to study curves satisfying Condition \eqref{eq:geom_cc}. This comes as no surprise, as Condition \eqref{eq:geom_cc} coincides with the condition coming from Dogra's quadratic RoS Chabauty method, which also involves all primes above $p$. However, the present generalization of the geometric quadratic Chabauty method does not make use of restriction of scalars in the same way as the RoS methods of Siksek and Dogra. 
While they use Weil restrictions to reduce the geometric situation to working over $\Q$, we work directly over $K$, and even integrally over $\oh_K$. Only at the end of the argument do we apply a restriction of scalars and work with all primes above $p$ simultaneously.  
More conceptually, the method of this paper is equivalent to applying the general geometric method described in Section \ref{s:overview} to the situation over $\Q$ with $X_\Q:=\Res_{K/\Q}(C_K)$, $J_\Q:=\Res_{K/\Q}(J_K)$, $Q_{\Q}:=\Res_{K/\Q}(T_K)$, $A_{\Q}:=\Res_{K/\Q}(J_K)$, $B_{\Q}:=\Res_{K/\Q}(J_K)^{\vee, \rho-1}$, and $F_{\Q}:=\Res_{K/\Q}(P_K^{\times})^{\rho-1}$. Here $T_K$ is the special torsor mentioned in Section \ref{s:HGC} and constructed in Section \ref{s:consT}, and $P_K^\times$ denotes the Poincar\'e $\G_m$-biextension of $(J_K, J_K^\vee)$ (see Section \ref{S:PoincareTorsor}). In contrast, a perhaps more natural approach to geometric quadratic Chabauty via restriction of scalars would have been to consider $F_K$ to be the Poincar\'e $\G_m$-biextension of $(\Res_{K/\Q}(J), \Res_{K/\Q}(J_K)^\vee)$, or a power thereof. This was the approach originally suggested to us by Edixhoven and Lido. However, it seems difficult in this case to produce the diagrams \eqref{diag:AJliftQK} and \eqref{diag:AJliftUQ1} required for the geometric method of Section \ref{s:overview}.
\end{remark}

\begin{remark}[$p$-saturation] 
 In \cite[Theorem 4.1]{spelierhashimoto}, Hashimoto and Spelier prove that (a linear version of) geometric Chabauty strictly outperforms Chabauty--Coleman, at least theoretically, although it is more difficult to implement for explicit examples. The reason is essentially that Chabauty--Coleman computes the $p$-saturation of the set computed in geometric (linear) Chabauty. Similarly, the geometric quadratic Chabauty method (both over $\Q$ and over number fields) is likely to remain unaffected by questions of $p$-saturation.
\end{remark}

\subsection{Outline}

In Section \ref{s:torsor}, we recall the necessary background on the Poincar\'e torsor, from which we build the torsor $T_K$ over $J_K$. We then spread out the geometry from $\spec K$ to $\spec \oh_K$. In Section \ref{s:consT}, we construct the torsor $\TT$. Section \ref{s:method} makes the strategy of the geometric quadratic Chabauty method precise. We then state the main technical results of this article and discuss how the condition \eqref{eq:geom_cc} arises.  
In Section \ref{s:closure}, which is the technical core of the paper, we parametrize the $p$-adic closure $\Y_t$ of the rational points $\TT(\oh_K)_t$ by a $p$-adic interpolation argument. We complete the proof of the main theorem in Section \ref{s:end}. Some questions are discussed in Section \ref{s:question}, including Question \ref{q1} raised above.

\subsection{Notation} \label{ss:notation}

For the convenience of the reader we provide a list of notations used in the main body of the paper. We have chosen to use notations similar to \cite{EL19} in order to facilitate the comparison with the original approach over $\Q$.

\begin{itemize}
\item $K/\Q$ is a number field of degree $d=r_1+2r_2$, where $r_1$ and $r_2$ are respectively the number of real embeddings and pairs of complex embeddings of $K$. 
\item $\oh_K$ denotes the ring of integers of $K$. 
\item $\delta=r_1+r_2-1$ is the rank of the unit group $\oh_K^\times$. 
\item $h=\cl(\oh_K)$ is the class number of $\oh_K$. 
\item $C_K$ is a smooth, proper, and geometrically connected curve over $K$ of genus $g\geq 2$. 
\item $J_K=\Pic^0_{C_K/K}$ is the Jacobian of $C_K$. It is an abelian variety of dimension $g$ over $K$.
\item $J^{\vee}_{K}=\Pic^0_{J_K/K}$ is the dual abelian variety of $J_K$.
\item $P^\times_K\lra J_K\times J_K^\vee$ is the Poincar\'e torsor. It is a biextension of $J_K$ and $J_K^\vee$ by $\G_m$.
\item $\CC$ is a regular proper model of $C_K$ over $\oh_K$ that we fix in this article.
\item $\J$ is the N\'eron model of $J_K$ over $\oh_K$.
\item $\J^\vee$ is the N\'eron model of $J_K^\vee$ over $\oh_K$.
\item $\Jo$ is the fiber-wise connected component of $0$ of $\J^\vee$.
\item $\mathbf{P}^\times\lra \J\times \Jo$ is the unique biextension of $\J$ and $\Jo$ by $\G_m$ whose base change to $K$ is the Poincar\'e torsor.
\item $j_b : C_K\lra J_K$ is the Abel--Jacobi map associated to a choice of point $b\in C_K(K)$.
\item $r=\rank_{\Z} J_K(K)$ is the Mordell--Weil rank over $K$.
\item $\rho=\rank_{\Z} \NS_{J_K/K}(K)$ is the rank of the N\'eron--Severi group of $J_K$ over $K$.
\end{itemize}

 \addtocontents{toc}{\protect\setcounter{tocdepth}{1}}


\section{The Poincar\'e biextension} \label{s:torsor}

We define the key geometric object studied in this article, namely the Poincar\'e torsor, along with its biextension structure, and explain how to spread out the geometry to $\spec \oh_K$.

\subsection{The Poincar\'e biextension over $K$}


\subsubsection{The Poincar\'e bundle}\label{S:PicBun}

Let $C_K$ be a smooth, proper, geometrically connected curve of genus $g\geq 2$ defined over $K$ with $C_K(K)\neq \emptyset$. Let $J_K:=\Pic_{C_K/K}^0$ be its Jacobian, that is, the connected component of the identity of the Picard scheme $\Pic_{C_K/K}$. This is an abelian variety of dimension $g$ defined over $K$. We denote its zero section by $0\in J_K(K)$, or alternatively by $e : \spec K \lra J_K$. Consider the Picard scheme $\Pic_{J_K/K}$ over $K$ as the contravariant functor from the category of $K$-schemes to abelian groups given by 
\begin{equation}\label{pic}
T \mapsto  \Pic(J_K \times T)/\pr_T^* \Pic(T),
\end{equation}
where $\pr_T : J_K \times T\lra T$ is the base-change of the structure morphism $J_K\lra \spec K$.  Denote by $J_K^\vee:=\Pic^0_{J_K/K}$ the  dual abelian variety of $J_K$,  
which comes equipped with a canonical principal polarization $\lambda : J_K\overset{\sim}{\lra} J_K^\vee$, given by translating the theta divisor. 
The functor described by \eqref{pic} is isomorphic to the functor given by 
\[
T\mapsto \{ \text{ isomorphism classes of rigidified line bundles } (L, \alpha) \text{ on } J_K\times T \: \}.
\]
Here a rigidification of the line bundle $L$ is an isomorphism $\alpha : \oh_T \overset{\sim}{\lra} e_T^* L$, where the section $e_T : T\lra J_K\times T$ is the one induced by $e$. Let $(P_K, \nu)$  denote the  universal rigidified line bundle over $J_K \times \Pic_{J_K/K}$. It satisfies the following universal property: if $(L, \alpha)$ is a rigidified line bundle on $J_K\times T$ along the zero section $e$, then there is a unique morphism $g : T \lra \Pic_{J_K/K}$ such that $(L, \alpha) \simeq (\id_{J_K}\times g)^*(P_K, \nu)$. 
The Poincar\'e bundle of $J_K$ is the restriction of the universal line bundle $P_K$ to $J_K\times J_K^\vee$ equipped with its canonical rigidification $\nu$, which we denote again by $P_K$ by a slight abuse of notation. 
The canonical rigidification of the Poincar\'e bundle gives rise to an isomorphism 
$\nu : \oh_{J_{K}^\vee}\overset{\sim}{\lra} P_K \vert_{\{ 0 \} \times J_K^\vee}.$
 Let $0$ denote the identity of the abelian variety $J_K^\vee$. There is a unique rigidification 
$
\nu' : \oh_{J_K} \overset{\sim}{\lra} P_K \vert_{J_K \times \{ 0 \}},
$
such that $\nu$ and $\nu'$ agree at the origin $(0, 0)$ in $(J_K\times J_K^\vee)(K)$. As a consequence, $(P_K, \nu, \nu')$ is a birigidified line bundle on $J_K\times J_K^\vee$ with respect to the identity elements.


\subsubsection{The Poincar\'e torsor}\label{S:PicTors}

Recall that given a line bundle $L$ on a scheme $X$, its associated $\Gm$-torsor is $L^\times := \textbf{Isom}_X (\oh_X, L)$, which is equipped with a free and transitive  action of $\Gm$.  
Note that $L^\times$ is Zariski locally trivial. In particular, it is represented by a scheme over $X$ which we again denote by $L^\times$ by slight abuse of notation (concretely, $L^\times$ is locally obtained by deleting the zero section of $L$). The Poincar\'e torsor $P_K^\times$ is the $\G_m$-torsor on $J_K\times J_K^\vee$ associated to the Poincar\'e bundle $P_K$.  Again, we  denote by $P_K^\times$ the scheme represented by the Poincar\'e torsor and  denote by $j_K : P_K^\times \lra J_K\times J_K^\vee$ the structure morphism. The torsor $P_K^\times$ inherits the compatible birigidification over $J_K\times \{ 0 \}$ and $\{ 0 \}\times J_K^\vee$ coming from $P_K$.

\begin{remark} 
Note that  $\G_m$-torsors on $X$ are classified by the \v{C}ech cohomology $\check{H}^1(X, \G_m)$, thus the operation $L \mapsto L^{\times}$ gives the inverse $\Pic(X)\lra \check{H}^1(X, \G_m)$ of the canonical isomorphism $\check{H}^1(X, \G_m)\lra H^1(X, \G_m) \simeq \Pic(X)$. In particular, every $\G_m$-torsor on $X$ arises in this way, and $\Pic(X)$ classifies isomorphism classes of $\G_m$-torsors on $X$. 
\end{remark}

\subsubsection{The Poincar\'e biextension}\label{S:PoincareTorsor}


The birigidified torsor $P_K^\times$ admits a unique compatible structure of $\G_m$-biextension of the pair $(J_K, J_K^\vee)$ (see \cite[VII Definition 2.1 \& Exemple 2.9.5]{SGA7}). Let us briefly explain what this means, without repeating the technical definition from SGA 7. 


\begin{itemize}
  \item {\bf Partial composition $+_1$:} 
We may view $P_K^\times$ as a scheme over $J_K^\vee$ via the structure morphism $\pr_2\circ j_K$. As such, $P_K^\times$ becomes a commutative $J_K^\vee$-group scheme which is an extension of $J_{K, J_K^\vee} := J_K \times J_K^\vee$ by $\G_{m, J_K^\vee} = \Gm \times J_K^\vee$. In other words, $P_K^\times$ fits into the following short exact sequence of  $J_K^\vee$-group schemes
    \begin{equation}\label{ext1}
        1 \lra \G_{m, J_K^\vee} \lra P_K^\times \lra J_{K, J_K^\vee} \lra 0.
    \end{equation}
To wit, let $S$ be a $K$-scheme, $y \in J_K^\vee(S)$ be an $S$-point of $J^\vee_K$, and $x_1, x_2 \in J_K(S)$ be two $S$-points of $J_K$. Let $z_1, z_2 \in P_K^\vee (S)$ be two $S$-points lying above $(x_1, y)$ and $(x_2, y)$ respectively via the structure map $j_K$. This group structure can be described as follows. 
The data of the points $z_1$ and $z_2$ is equivalent to the data of two nowhere vanishing sections $\alpha_1, \alpha_2 \in (x_1, y)^* P_K (S)$ of the pullback of the Poincar\'e bundle. As part of the requirement of being a $\G_m$-biextension, we have an isomorphism of line bundles over $\oh_S$
    \begin{equation}\label{eq:thm_square}
(x_1, y)^* P_K \otimes (x_2, y)^* P_K \simeq (x_1+x_2, y)^* P_K,
\end{equation} supplied in this case by the theorem of the cube. 
Under this canonical isomorphism, the tensor product  $\alpha_1 \otimes \alpha_2$ corresponds to a nowhere zero section $\alpha_3$ of $(x_1+x_2, y)^* P_K$, thus producing a point $z_3 \in P_K^\times(S)$, which lies above the point  $(x_1+ x_2, y)$ of $J_K \times J_K^\vee$. The commutativity of $P_K^\times$ as a $J_K^\vee$-group is clear, as well as the exact sequence \eqref{ext1}. 
    We denote by $+_1$ the resulting partial composition law on $P_K^\times$, which provides the group structure of $P_K^\times$ over $J_K^\vee$, but not over $K$. In other words, it is defined on pairs of points $z_1, z_2\in P_K^\times(S)$ such that $$ \pr_2(j_K(z_1))=\pr_2(j_K(z_2)).$$
   We also denote the group structure on the $J_K^\vee$-group scheme $J_{K, J_K^\vee}$ by $+_1$, again slightly abusing notations. The partial composition law $+_1$ on $P_K^\times$ then satisfies
\[
 z_1 +_1 z_2 \in P_K^\times (S) \:\:  \longmapsto \: (x_1, y) +_1 (x_2, y) = (x_1+x_2, y) \in J_{K, J_K^\vee} (S). 
\]
    
\item {\bf Partial composition $+_2$:} 
On the other hand, we may view $P_K^\times$ as a  $J_K$-scheme via the structure morphism $\pr_1\circ j_K$. As above, this makes  $P_K^\times$ into an extension of $J^\vee_{K, J_K}$ by $\G_{m, J_K}$, 
    which fits into a short exact sequence of commutative $J_K$-group schemes
    \begin{equation}\label{ext2}
        1 \lra \G_{m, J_K} \lra P_K^\times \lra J^\vee_{K, J_K} \lra 0.
    \end{equation} 
   We denote by $+_2$ the resulting partial composition law on $P_K^\times$, this time defined on couples of points $z_1, z_2 \in P_K^\times(S)$ that satisfy 
    \[ \pr_1(j_K(z_1))=\pr_1(j_K(z_2)). \]
    
 \item {\bf Compatibility:}    
    The commutative group scheme extensions (\ref{ext1}) and (\ref{ext2}) are compatible in the following sense. Let $S$ be any $K$-scheme. Let $z_\alpha, z_\beta, z_\gamma, z_\delta \in P_K^\times (S)$ be arbitrary   $S$-points such that 
    $$j_K(z_\alpha) = (x_1, y_1), \quad j_K(z_\beta) = (x_1, y_2), \quad j_K(z_\gamma) = (x_2, y_1), \quad j_K(z_\delta) = (x_2, y_2),$$
     for some $S$-points $x_1, x_2\in J_K(S)$ and $y_1, y_2\in J_K^\vee(S)$. Then
    \begin{equation}
        (z_\alpha +_2 z_\beta)+_1 (z_\gamma +_2 z_\delta)=(z_\alpha+_1 z_\gamma)+_2 (z_\beta +_1 z_\delta).
    \end{equation}
We summarize this compatibility in the following picture for the convenience of the reader: 

\vspace{2mm}
\begin{center}
\tikzset{every picture/.style={line width=0.7pt}} 
\resizebox{7.5cm}{!}{
\begin{tikzpicture}[x=0.7pt,y=0.65pt,yscale=-1,xscale=1]
\draw  [dash pattern={on 4.5pt off 4.5pt}][line width=0.75]  (195.6,127.81) -- (282.39,127.81) -- (282.39,211.17) -- (195.6,211.17) -- cycle ;
\draw    (168.31,267.88) -- (307.65,268.28) ;
\draw    (426.62,118.03) -- (425.81,228) ;
\draw    (242.46,222.88) -- (242.46,260.8) ;
\draw [shift={(242.46,262.8)}, rotate = 270] [color={rgb, 255:red, 0; green, 0; blue, 0 }  ][line width=0.75]    (10.93,-3.29) .. controls (6.95,-1.4) and (3.31,-0.3) .. (0,0) .. controls (3.31,0.3) and (6.95,1.4) .. (10.93,3.29)   ;
\draw    (341.88,170.86) -- (407.51,171.62) ;
\draw [shift={(409.51,171.65)}, rotate = 180.66] [color={rgb, 255:red, 0; green, 0; blue, 0 }  ][line width=0.75]    (10.93,-3.29) .. controls (6.95,-1.4) and (3.31,-0.3) .. (0,0) .. controls (3.31,0.3) and (6.95,1.4) .. (10.93,3.29)   ;

\draw (173.97,110.59) node [anchor=north west][inner sep=0.75pt]   [align=left] {$\displaystyle z_{\alpha }$};
\draw (171.62,204.52) node [anchor=north west][inner sep=0.75pt]   [align=left] {$\displaystyle z_{\beta }$};
\draw (286.52,204.5) node [anchor=north west][inner sep=0.75pt]   [align=left] {$\displaystyle z_{\delta }$};
\draw (286.43,110.5) node [anchor=north west][inner sep=0.75pt]   [align=left] {$\displaystyle z_{\gamma }$};
\draw (174.79,275) node [anchor=north west][inner sep=0.75pt]    {$x_{1}$};
\draw (284.8,275) node [anchor=north west][inner sep=0.75pt]    {$x_{2}$};
\draw (444.52,113.34) node [anchor=north west][inner sep=0.75pt]    {$y\mathnormal{_{1}}$};
\draw (444.52,205.71) node [anchor=north west][inner sep=0.75pt]    {$y_{2}$};
\draw (245.54,234.24) node [anchor=north west][inner sep=0.75pt]    {$\mathrm{pr_{1}} \circ j_{K}$};
\draw (347.47,150.91) node [anchor=north west][inner sep=0.75pt]    {$\mathrm{pr_{2}} \circ j_{K}$};
\draw (320.1,257.96) node [anchor=north west][inner sep=0.75pt]  [font=\large]  {$J_K$};
\draw (412,243) node [anchor=north west][inner sep=0.75pt]  [font=\large]  {$J_K^\vee $};
\draw (138.95,158.56) node [anchor=north west][inner sep=0.75pt]  [font=\small,color={rgb, 255:red, 208; green, 2; blue, 27 }  ,opacity=1 ] [align=left] {$\displaystyle z_{\alpha }\!+_{2}\! z_{\beta }$};
\draw (284.1,159.91) node [anchor=north west][inner sep=0.75pt]  [font=\small,color={rgb, 255:red, 208; green, 2; blue, 27 }  ,opacity=1 ] [align=left] {$\displaystyle z_{_{\gamma }}\!+_{2}\! z_{\delta }$};
\draw (213.02,111.59) node [anchor=north west][inner sep=0.75pt]  [font=\small,color={rgb, 255:red, 208; green, 2; blue, 27 }  ,opacity=1 ] [align=left] {$\displaystyle z_{\alpha }\!+_{1}\! z_{\gamma }$};
\draw (214.02,193.78) node [anchor=north west][inner sep=0.75pt]  [font=\small,color={rgb, 255:red, 208; green, 2; blue, 27 }  ,opacity=1 ] [align=left] {$\displaystyle z_{\beta }\!+_{1}\! z_{\delta }$};
\draw (216.11,274) node [anchor=north west][inner sep=0.75pt]  [font=\small,color={rgb, 255:red, 208; green, 2; blue, 27 }  ,opacity=1 ] [align=left] {$\displaystyle x_{1} \!+ \!x_{2}$};
\draw (440.2,162.47) node [anchor=north west][inner sep=0.75pt]  [font=\small,color={rgb, 255:red, 208; green, 2; blue, 27 }  ,opacity=1 ] [align=left] {$\displaystyle y_{1} +y_{2}$};
\end{tikzpicture}
}
\end{center}
\end{itemize}

\subsubsection{Action of $\G_m$}

We briefly describe the action of $\Gm$ on the Poincar\'e torsor. To this end, we let $e_{J_K}\in \Hom_{J_K}(J_K, P_K^\vee)$ and $e_{J_K^\vee}\in \Hom_{J_K^\vee}(J_K^\vee, P_K^\times)$ denote the identity sections of $P_K^\times$ viewed as a $J_K$-group scheme and a $J_K^\vee$-group scheme respectively. 
Restricting the short exact sequence (\ref{ext1}) of commutative $J_K^\vee$-group schemes via the identity section $\spec K \lra J_K^\vee $ yields a short exact sequence of commutative $K$-group schemes 
\[\begin{tikzcd}[bend angle = 20]
1 \arrow{r}
& \G_{m, K} \arrow{r}
& P^\times_{K} \vert_{J_K\times \{ 0 \}} \arrow{r}
& J_K \arrow{r} \arrow[bend right,end anchor={[shift={(0pt,-6pt)}]}]{l}[swap]{e_{J_K}}
& 0,
\end{tikzcd}\]
which is split by the section $e_{J_K}$. In particular, we have 
$ P^\times_{K} \vert_{J_K\times \{ 0 \}}=\G_{m, J_K}=\G_{m, K}\times J_K$. By a similar reasoning using the identity section $e_{J_K^\vee}$, we also have 
$P^\times_{K} \vert_{\{ 0 \} \times J_K^\vee}=\G_{m, J_K^\vee}. $
These canonical splittings  allow for a useful description of the $\G_m$-action on $P^{\times}_K$ in terms of the partial group laws $+_2$ and $+_1$. For a $(J_K\times J_K^\vee)$-scheme $S$, consider $t \in P_K^{\times}(S)$ and $u \in \G_m(S)$, and let $(x, y)$ be the image of $t$ in $(J_K \times J_K^{\vee})(S)$. Consider a point $v=v_{x, u} \in P_K^{\times}(S)$ lying over $(x, 0)$ and corresponding to $(u, 0)$ under the identification ${P_K^{\times}}\vert_{J_K \times \{0\}}(S) \simeq \G_m(S) \times J_K(S)$. 
The action of $u$ on the point $t$ is then given by
\begin{equation}\label{rem:internal-Gm-action}
u \cdot t=v+_2 t.
\end{equation}
The point $v_{x, u}$ does not depend on $t$, only on $x$ and $u$. The change of $v_{x, u}$ in the parameter $x$ is described by the relative group law $+_1,$ namely $v_{x_1+x_2, u}=v_{x_1, u}+_1v_{x_2, u}$. Similarly, we have $v_{x, u_1u_2}=v_{x, u_1}+_2v_{x, u_2}$.

Instead of using the point $(x, 0)$, one could work with $(0, y)$ and the operation $+_1$. These two points of view are equivalent by the compatibility between $+_1$ and $+_2$. As a consequence, the $\G_m$-action commutes with the operations $+_1$ and $+_2$ in the following sense. Given two points $a, b \in P_K^{\times}(S)$ lying over points of the form $(x, *)$ in $J_K\times J_K^{\vee}(S)$, 
and $u, u' \in \G_m(S),$ we have
\begin{equation}\label{GmAndGrpLawsCommute}
\begin{aligned}
(u \cdot a)+_2(u' \cdot b)& = (v_{x, u}+_2 a)+_2(v_{x, u'}+_2 b)\\
& =(v_{x, u}+_2v_{x, u'})+_2(a+_2b) \\
& =(uu') \cdot (a+_2b).
\end{aligned}
\end{equation}
The same holds for $+_1$.

\subsection{Spreading out the geometry}

In the method of geometric quadratic Chabauty, it is crucial to spread out the geometry to $\oh_K$. Roughly speaking, the reason is that it is necessary to work with finitely generated $\Z$-modules. The module $\G_m(K)=K^\times$ is not finitely generated, whereas $\G_m(\oh_K)=\oh_K^\times$ is. Indeed, if $r_1$ and $r_2$ denote respectively the number of real embeddings and pairs of complex embeddings of $K$, then Dirichlet's Unit Theorem implies that 
$
\delta:=\rank_\Z \oh_K^\times = r_1+r_2-1.
$

\subsubsection{Models over $\oh_K$}\label{s:spread}

Let $\CC$ denote a regular proper model of $C_K$ over $\oh_K$. Let $\CC^{\sm}$ denote the smooth locus of $\CC$. 
By properness and regularity respectively, we have the identifications 
$$    C_K(K)=\CC(\oh_K)=\CC^{\sm}(\oh_K) .$$
Let $\J$ and $\J^{\vee}$ denote the N\'eron models over $\oh_K$ of $J_K$ and $J_K^{\vee}$ respectively. Denote by $\J^{\circ}$ and $\Jo$ the fiberwise connected components of $0$ in $\J$ and $\J^{\vee}$ respectively. The quotient $\J^\vee/\J^{\vee, \circ}$ is an \'etale group scheme over $\oh_K$ with finite fibers. 

Suppose that $C_K(K)$ is non-empty, and let $b\in C_K(K)$ be a fixed  rational point. Such a choice yields an Abel--Jacobi embedding $j_b : C_K\hookrightarrow J_K$, by sending a point $x$ to the linear equivalence class of the divisor $(x)-(b)$.
The map $j_b$ extends uniquely to a morphism 
   $$ j_b: \CC^{\text{sm}}\lra \J $$
over $\oh_K$ by the N\'eron Mapping Property. The extension of the Poincar\'e biextension to $\spec \oh_K$ is supplied by work of Grothendieck. 

\begin{proposition} 
The Poincar\'e torsor $P_K^{\times}$ extends uniquely to a biextension $\PP^\times$ of $(\J, \Jo)$ by $\G_m$. In particular, given an $\oh_K$-scheme $S$ and two points $(x, y), (x, y') \in \J\times \Jo(S)$, there is an isomorphism 
\begin{equation}\label{isom:cube}
 (x, y)^* \PP \otimes (x, y')^* \PP \simeq (x, y+ y')^* \PP,
 \end{equation}
where $\PP$ is the line bundle over $\J \times \Jo$ corresponding to $\PP^\times$. 
\end{proposition} 

\begin{proof}
This is \cite[VIII. Theorem 7.1(b) \& Remark~7.2]{SGA7}. 
\end{proof}

We denote the structure morphism of this $\G_m$-torsor by 
$j : \PP^{\times}\lra \J\times \Jo. $
The uniqueness of the extension follows from the connectedness of $\Jo$. We remark that the commutative group scheme extension structures and their compatibilities discussed in Section \ref{S:PoincareTorsor} extend to the integral version $\PP^\times$.

\subsubsection{Integral points on the Poincar\'e torsor}\label{S:lifting}

We show how to lift certain integral points on $\J \times \Jo$ across the structure map $j: \PP^\times \lra \J \times \Jo$. 
Let $(x, y)$ be an $\oh_K$-point of $\J\times \Jo$, and let $(x, y)^* \PP^{\times}$ be the pull-back of $\PP^{\times}$ to $\oh_K$, as pictured in the Cartesian diagram
\begin{equation}\label{diag:lift}
\begin{tikzcd}
(x, y)^* \PP^\times\arrow{r}\arrow{d} \arrow[dr, phantom, "\square"]
  & \PP^\times  \arrow{d} \\
  \spec \oh_K \arrow{r}{(x, y)}
  & \J\times  \Jo.
\end{tikzcd}
\end{equation}
Lifting the point $(x, y)$ to $\PP^\times$ amounts to finding a section of the torsor $(x, y)^* \PP^\times \lra \spec \oh_K$. 

In the case $K=\Q$, all $\Gm$-torsors are trivial over $\spec \Z$ and thus admit a section over $\Z$, which is unique up to $\G_m(\Z)=\{ \pm 1 \}$. A lift of the integral point $(x, y)$ to $\PP^\times$ therefore always exists. 
 In the case of a general number field $K$, it is not always possible to lift an $\oh_K$-point $(x, y)$ of $\J\times \Jo$ to $\PP^\times$ when the class number $h$ of $K$ is non-trivial. However, the previous argument carries over to $\oh_K$-points of the form $(x, h\cdot y)$. 
\begin{lemma}\label{lem:h}
Any $\oh_K$-point of $\J\times \Jo$ of the form $(x, h\cdot y)$ with $(x, y)\in \J\times \Jo (\oh_K)$ admits a lift to an $\oh_K$-point of the Poincar\'e torsor $\PP^\times$. This lift is unique up to multiplication by an element of $\oh_K^\times$.
\end{lemma}

\begin{proof}
We repeatedly apply (\ref{isom:cube}) 
to obtain an isomorphism
$
    ((x, y)^* \PP)^{\otimes h} \simeq (x, h\cdot y)^* \PP
$
of line bundles over $\spec \oh_K$. It follows that $(x, h\cdot y)^* \PP^\times$ is trivial as a $\G_m$-torsor over $\oh_K$, since $\Pic (\oh_K)$ has size $h$. 
\end{proof}


\section{Construction of the torsor $\TT$}\label{s:consT}

The goal of this section is to construct a certain $\G_m^{\rho-1}$-torsor $\TT$ over $\J$ along with a lift of the Abel--Jacobi map $j_b : \CC^{\sm} \lra \J$. Here $\rho$ denotes the rank of the N\'eron-Severi group of $J_K$. We begin by constructing the corresponding torsor $T_K$ over $J_K$, and then proceed to spread out the geometry. Once the torsor $\TT$ has been defined, we construct the desired lift of the Abel--Jacobi map.

\subsection{Trivialization of the Poincar\'e torsor} \label{trivialization}

Let  $\lambda : J_K\overset{\sim}{\lra} J_K^\vee$ be the canonical principal polarization defined in Section \ref{S:PicBun}. By functoriality of $\Pic$, we have a commutative diagram of commutative $K$-group schemes with exact rows
\begin{equation}\label{diag:NS}
\begin{tikzcd}[bend angle = 20]
  0 \arrow{r}
  & J_K^\vee \arrow{r} \arrow{d}{-\lambda^{-1}}[swap]{\wr}
  & \Pic_{J_K/K} \arrow{r}{\pi} \arrow{d}{j_b^*}
  & \NS_{J_K/K} \arrow{r} \arrow{d}{j_{b, \NS}^*}
  & 0 \\
  0 \arrow{r}
  & J_K \arrow{r}
  & \Pic_{C_K/K} \arrow{r}{\deg}
  & \Z_K \arrow{r}
  & 0.
\end{tikzcd}
\end{equation}
Here $\NS_{J_K/K}$ denotes the N\'eron--Severi group scheme of $J_K$, i.e., the \'etale $K$-group scheme of components of the Picard scheme associated to $J_K$. Note that we have used the fact that the map induced by $j_b$ on $\Pic^0$ via pull-back agrees with $-\lambda^{-1}$. It is in particular an isomorphism. 

Let $\Hombf(J_K, J_K^\vee)^{+} \subset \Hombf(J_K, J_K^\vee)$ denote the closed subgroup scheme of self-dual homomorphisms. We refer to \cite[Proposition 7.14 \& \S 7.18]{EGM} for questions of representability. There is a map 
$$
    \varphi : \Pic_{J_K/K}\lra \Hombf(J_K, J_K^\vee)^+, 
$$
defined by sending the class of a line bundle $L$ to the map $\varphi_L$, which maps a closed point $x\in J_K$ to $[\tr_x^* L\otimes L^{-1}]$, where $\tr_x : J_K\lra J_K$ denotes the translation by $x$. The kernel of $\varphi$ is equal to $\Pic^0_{J_K/K}=J_K^\vee$, and the map $\varphi$ induces an isomorphism of $K$-group schemes \cite[Corollary 11.3]{EGM}
\begin{equation}\label{map:phitilde}
    \tilde{\varphi} : \NS_{J_K/K}\overset{\sim}{\lra} \Hombf(J_K, J_K^\vee)^+.
\end{equation}

\begin{definition} 
At the level of $K$-points, we define the group $\Hom(J_K, J_K^\vee)_0^+$ to be 
$$
\Hom(J_K, J_K^\vee)_0^+ := \ker \left( j_{b, \NS}^*\circ \tilde{\varphi}^{-1}: \Hom(J_K, J_K^\vee)^+ \lra \Z \right). 
$$ 
\end{definition}

\begin{proposition}\label{prop:trivtorsor}
For all $f\in \Hom(J_K, J_K^\vee)_0^+$, there exists a unique element $c_f\in J_K^\vee(K)$ with the property that the $\G_m$-torsor
$
j_b^*(\id, \tr_{c_f}\circ f)^* P_K^\times
$
over $C_K$ is trivial.
In particular, for all $n\in\Z_{\ge 1}$, its $n^\text{th}$ power $j_b^*(\id, n\cdot \circ \tr_{c_f}\circ f)^* P_K^\times$ is also trivial. 
\end{proposition}

\begin{proof}
At the level of $\overline{ K}$-points, the diagram (\ref{diag:NS}) gives rise to the commutative diagram
\begin{equation}\label{diag:NSK}
\begin{tikzcd}[bend angle = 20]
& & &  \Hom(J_K, J_K^\vee)_0^+ \arrow[d, hook] \arrow[ddl, dashed, bend right, start anchor={[shift={(0pt,-3pt)}]}, end anchor={[shift={(3pt,0pt)}]}, near start, "s_1"] \arrow[ddl, dashed, bend right, swap, near start, "s_2"] \\
&  & \ker (j_{b,\overline{ K}}^*) \arrow[r, swap, "\sim"] \arrow[hook]{d}
  & \ker (j_{b, \overline{K}, \NS}^*) \arrow[hook]{d} 
\\
  0 \arrow{r}
  & J_K^\vee(\overline{K}) \arrow{r} \arrow{d}{-\lambda^{-1}}[swap]{\wr}
  & \Pic(J_{\overline{K}}) \arrow{r}{\pi} \arrow{d}{j_{b, \overline{K}}^*}
  & \NS_{J_K/K}(\overline{K}) \arrow{r} \arrow{d}{j_{b,\overline{K}, \NS}^*}
  & 0 \\
  0 \arrow{r}
  & J_K(\overline{K}) \arrow{r}
  & \Pic(C_{\overline{K}}) \arrow{r}{\deg}
  & \Z \arrow{r}
  & 0.
\end{tikzcd}
\end{equation}
The map $\pi$ in the first  short exact sequence of this diagram admits two splittings when restricted to $ \Hom(J_K, J_K^\vee)_0^+$, which is viewed as a subgroup of $\ker (j_{b, \overline{K}, \NS}^*)$ via the map $\tilde \varphi^{-1}$. The first section 
 $$s_1 : \Hom(J_K, J_K^\vee)^+ \lra \Pic(J_{\overline K})$$
  is defined by mapping a self-dual homomorphism $f$ defined over $K$ to the isomorphism class of the $\G_m$-torsor $L_f^\times:=(\id, f)^* P_K^\times$ on $J_K$, which is an element of $\Pic (J_K) \subset \Pic (J_{\overline{K}})$. 
We observe, by\cite[Proposition 11.1]{EGM}, that
$$
\tilde{\varphi}\circ \pi \circ s_1(f)=\varphi_{L_f}=f+f^\vee=2f. 
$$
 The second splitting is given by inverting $\pi$ on $\ker (j_{b, \overline{K}}^*)$, i.e., by 
 $$
s_2 :\Hom(J_K, J_K^\vee)_0^+  \hookrightarrow \ker (j_{b, \overline{K}, \NS}^*)  \overset{\pi^{-1}}{\lra} \ker (j_{b, \overline{K}}^*) \subset \Pic(J_{\overline{K}}).$$
Again the image of $s_2$  lies in $\Pic(J_K).$ Given $f\in \Hom(J_K, J_K^\vee)_0^+$ we define 
$$
c_f:=2s_2(f)-s_1(f)\in \Pic (J_K).
$$
As $ c_f \in \ker (\pi)$, we thus have $c_f \in J_K^{\vee} (K)$. We observe that for a line bundle $L$ on $J_K$ corresponding to a closed point $x \in J_K^\vee$,  we have
$$(\id, f)^* \Big( (\id \times \tr_{x})^* P_K  \Big) \simeq (\id, f)^* \big( P_K \otimes \textup{pr}_1^* L \big) \simeq (\id, f)^* P_K \otimes L,$$
where $\textup{pr}_1$ is the projection $J_K \times J_K^\vee \lra J_K$. 
By construction, $c_f$ is therefore the unique element in $J_K^\vee(K)$ such that
$$
s_1(f)+c_f=[(\id, \tr_{c_f}\circ f)^* P_K^\times]\in \ker j_{b}^*.
$$
\end{proof}

\begin{remark}\label{remark:Chow_Heegner}
Proposition \ref{prop:trivtorsor} defines a homomorphism $\Hom(J_K, J_K^\vee)_0\lra J_K^\vee(K)$ mapping $f$ to the point $c_f$. The association takes the line bundle $L_f:=(\id, f)^*P_K$ on $J_K$ and pulls its back to the degree $0$ line bundle $j_{b}^*L_f$ on $C_K$. The latter is naturally an element of $J_K(K)$. Then $c_f:=\lambda(j_{b}^*L_f)\in J_K^\vee(K)$. Indeed, $L_f+c_f\in \Pic(J_K)$ satisfies $j_b^*(L_f+c_f)=j_b^*L_f-\lambda^{-1}(c_f)=0$. With this description in hand, it becomes apparent that $c_f$ is a Zhang point, as defined in \cite[Definition 3.3.1]{daubthesis}. In the notation therein, $c_f=P(L_f, \lambda, \id_{J_K}, b)$. It follows that $2c_f$ is a Chow--Heegner point associated with the modified diagonal cycle on $(C_K)^3$ based at $b$ \cite[Proposition 3.3.3]{daubthesis}. For details about Chow--Heegner points attached to diagonal cycles we refer to \cite{drs, ddlr}. Such points naturally arise in the context of bounding rational points on modular curves in \cite{dogralefourn}. Computing the points $c_f$, which is required for any explicit implementation of the geometric method of this paper, is of independent interest.    
\end{remark}

The group $\NS_{J_K/K}(K)$ is a finitely generated free $\Z$-module whose rank is denoted by $\rho$ and called the Picard number of $J_K$. The kernel 
$\ker (j_{b, \NS}^*: \NS_{J_K/K} (K) \lra \Z)$ is a free $\Z$-module of rank $\rho-1$, and so is the group $\Hom(J_K, J_K^\vee)_0^+$.

\begin{notation} \label{notation:f_i} 
\hfill
\begin{itemize} 
\item Let $f_1, \ldots, f_{\rho-1}$ be a basis of $\Hom(J_K, J_K^\vee)_0^+$. 
\item For each $i=1, \ldots, \rho-1$, let $c_i:=c_{f_i}\in J_K^\vee(K)$ be the element corresponding to $f_i$ in Proposition \ref{prop:trivtorsor}.
\item For each integer $n\in \mathbb{Z}_{\ge 1}$, denote by $\alpha_{n, i, K}$ the map 
$$
\alpha_{n, i, K}: J_K \xrightarrow{(\id, n\cdot \circ \tr_{c_i} \circ f_i)} J_K\times J_K^\vee.
$$
\end{itemize}
\end{notation}

\begin{definition}\label{cons:lifting_over_K} By Proposition \ref{prop:trivtorsor}, the pull-back $j_b^* \big( \alpha_{n, i, K}^* P_K^\times \big)$ is a trivial $\G_m$-torsor over $C_K$. 
In other words, it admits a section over $C_K$. This gives rise to a lift of $j_b$, unique up to $K^\times$, which we shall fix and denote by $\tilde{j}_{b}^{(n, i)}$. This is pictured in the diagram
\begin{equation}\label{diag:AJliftK}
\begin{tikzcd}[bend angle = 20]
  & \alpha_{n, i, K}^* P_K^\times  \arrow{r} \arrow{d} \arrow[phantom]{dr}{\square}
  & P_K^\times \arrow{d} \\
  C_K \arrow{r}{j_b} \arrow[dashrightarrow]{ru}{\tilde{j}_{b}^{(n, i)}}
  & J_K \arrow{r}{\alpha_{n, i, K}}
  & J_K\times J_K^\vee.
\end{tikzcd}
\end{equation}
\end{definition}

\subsection{Definition of the torsor}

We introduce and recall some notations, and refer the rest to Section \ref{s:spread}. Let $\mathfrak{n}$ be the product of prime ideals in $\oh_K$ such that $\CC$ is smooth away from $\spec(\oh_K/\mathfrak{n})$. Let $\Phi^\vee = \J^\vee / \Jo$ be the group scheme of connected components of $\J^\vee$. It is trivial outside $\oh_K/\mathfrak{n}$ with finite \'etale fibers over $\oh_K/\mathfrak{n}$. Let $m$ denote the least common multiple of the exponents of $\Phi^\vee(\bar{\F}_\mathfrak{q})$ over all prime ideals $\mathfrak{q}$ of $\oh_K$. Finally, recall that $h$ denotes the class number of $K$. 

By the N\'eron Mapping Property, for each $i\in \{ 1, \ldots, \rho-1 \}$, the maps
\[
\begin{cases}
f_i : J_K\lra J_K^\vee & \\
\tr_{c_i} : J_K^\vee \lra J_K^\vee & \\
hm\cdot : J_K^\vee \lra J_K^\vee &
\end{cases} \qquad \text{ extend uniquely to } \qquad
\begin{cases}
f_i : \J\lra \J^\vee & \\
\tr_{c_i} : \J^\vee \lra \J^\vee & \\
hm\cdot : \J^\vee \lra \J^\vee. &
\end{cases}
\]
The map $\alpha_{hm, i, K} : J_K\lra J_K\times J_K^\vee$ therefore extends uniquely to a map of $\oh_K$-schemes 
$$\alpha_{hm, i} = (\id, hm\cdot \circ \tr_{c_i} \circ f_i): \J\lra  \J\times \J^\vee.$$  
The integer $m$ is chosen so that the image of this map lies in $\J\times \Jo$. 

\begin{definition} \label{map:alpha} 
Taking the product over $i\in \{ 1, \ldots, \rho-1\}$, we obtain the $\oh_K$-morphism
\begin{equation*} 
    \alpha =(\id, ({hm\cdot} \circ \tr_{\underline{c}} \circ \underline{f})) 
    :=(\id, ({hm\cdot} \circ \tr_{c_i} \circ f_i)_{i=1}^{\rho-1}) : \J \lra \J\times (\Jo)^{\rho-1}.
\end{equation*}
\end{definition}

Consider the map $\PP^\times \lra \J \times \Jo \lra \J$ defined as the composition of the structure map $j$ with the first projection. Using this morphism, we form the $(\rho-1)$-fold self-product 
$$
    \PP^{\times, \rho-1}:= \PP^\times \times_{\J} \ldots \times_{\J} \PP^\times.
$$
We naturally have a morphism
$    \PP^{\times, \rho-1} \lra \J\times (\Jo)^{\rho-1},$
which endows $\PP^{\times, \rho-1}$ with the structure of a $\G_m^{\rho-1}$-torsor over $\J\times (\Jo)^{\rho-1}$.
This leads to the following key construction. 
\begin{definition}
Define the $\G_m^{\rho-1}$-torsor $\TT$ over $\J$ to be the pull-back of the $\G_m^{\rho-1}$-torsor $\PP^{\times, \rho-1}$ over $\J\times (\Jo)^{\rho-1}$ by the map $\alpha$, i.e.,
\begin{equation*}
    \TT:= \PP^{\times, \rho-1}\times_{\alpha} \J=\alpha^* \PP^{\times, \rho-1}  = 
    (\id, hm\cdot \circ \tr_{c_1} \circ f_1)^* \PP^\times\times_{\J} \ldots \times_{\J}(\id, hm\cdot \circ \tr_{c_{\rho-1}} \circ f_{\rho-1})^* \PP^\times.
\end{equation*}
\end{definition}

\subsection{Lifting the Abel--Jacobi map}

By taking the product over $i$ of the lifts $\tilde{j}_b^{(hm, i)}$ of Definition \ref{cons:lifting_over_K}, we obtain a lift $\tilde{j}_b$ of $j_b$ to $T_K:=\TT\times_J J_K$, as pictured in the commutative diagram
\begin{equation}\label{diag:AJliftTK}
\begin{tikzcd}[bend angle = 20]
  & T_K  \arrow{r} \arrow{d} \arrow[phantom]{dr}{\square}
  & P_K^{\times, \rho-1} \arrow{d} \\
  C_K \arrow{r}{j_b} \arrow[dashrightarrow]{ru}{\tilde{j}_{b}}
  & J_K \arrow{r}{\alpha_K}
  & J_K\times (J_K^{\vee, 0})^{\rho-1},
\end{tikzcd}
\end{equation}
where $\alpha_K$ denotes the base change of the map $\alpha$  to $K$.

The goal is to extend this diagram over $\oh_K$. However, lifting the map $j_b : \CC^{\sm}\lra \J$ to the torsor $\TT$ is generally not possible, the problem being that the fibers $C^{\sm}_{\F_\mathfrak{q}} := \CC^{\sm} \times_{\spec \oh_K} \spec {\F_{\mathfrak{q}}}$, for primes $\mathfrak{q}\vert \mathfrak{n}$, may contain too many components. To remedy this, we consider one geometrically irreducible component in each such fiber at a time.  

\begin{definition} \label{cons:U}
Let $\U \subset \CC^{\sm}$ be an open subscheme obtained by removing, for every $\mathfrak{q} \vert \mathfrak{n}$, all but one irreducible component of $C^{\sm}_{\F_\mathfrak{q}}$ that is furthermore geometrically irreducible. We will lift the map $j_b$ to a map $\tilde j_b^U: \U \lra \TT$ for each such open subscheme $\U$.   
\end{definition}

\begin{remark}\label{remark:U} We first remark that  such a subscheme $\U$ exists under the assumption that $C_K$ admits a $K$-rational point. Secondly, for the purpose of determining the set of rational points $C_K (K) = \CC^{\sm} (\oh_K)$, it suffices to consider subschemes of the form $\U$, as there are finitely many of them and each point in $\CC^{\sm}(\oh_K)$ lies in exactly one such $\U$. Both remarks follow from the following simple lemma. 
\end{remark}

\begin{lemma}
Let $X$ be an irreducible variety over a field $k$ that admits a smooth $k$-rational point. Then $X$ is geometrically irreducible. 
\end{lemma}
 
 
\begin{proof} Let $A = \Gamma(U, \oh_X)$ be the ring of functions on a normal affine open neighborhood $U$ of the smooth rational point. Then $A$ admits a map $A \lra k$ of $k$-algebras. Let $k'$ be the separable algebraic closure of $k$ in the function field $k(X)=\mathrm{Frac}(A)$. Since $U$ is normal, we have $k' \subset A$ which forces $k' = k$. This is equivalent to $X$ being geometrically irreducible by \cite[Corollaire~4.5.10]{EGA4}. 
\end{proof}

The construction of the desired lift of $j_b$ is analogous to the one in \cite[\S 2]{EL19}, except that we pull back $\PP^{\times}$ via morphisms of the form $(\id, {hm\cdot}\circ \tr_c\circ f):\J\lra \J\times \Jo,$ where in the second factor we incorporate an additional multiplication by the class number $h$ to ensure the existence of such a lift.

\begin{proposition}\label{prop:Ulift}
Let $\U$ be an open subscheme of $\CC^{\sm}$ as in Definition \ref{cons:U}. There exists a lift $\tilde{j}_b^U$ of $j_{b}\vert_{\U}$ to $\TT$, unique up to $\oh_K^{\times, \rho-1}$, which makes the following diagram commute:
\begin{equation}\label{diag:AJliftU}
\begin{tikzcd}[bend angle = 20]
  &
  & \TT  \arrow{r} \arrow{d} \arrow[phantom]{dr}{\square}
  & \PP^{\times, \rho-1} \arrow{d} \\
  \U \arrow[hook]{r} \arrow[dashrightarrow]{rru}{\tilde{j}_{b}^U}
  & \CC^{\sm} \arrow{r}{j_b} 
  & \J \arrow{r}{\alpha}
  & \J\times (\Jo)^{\rho-1}.
\end{tikzcd}
\end{equation}
\end{proposition}

\begin{proof}
The restriction of the torsor
$(\id, m\cdot \circ \tr_{c_i} \circ f_i)^* \PP^\times$ to $\U$ gives an element of $\Pic(\U)$, whose pull-back to $C_K$ equals 
$j_{b}^*\alpha_{m, i, K}^*P_K^{\times}$ and is trivial by Proposition \ref{prop:trivtorsor}. In other words, when restricted to $\U,$ the torsor $(\id, m\cdot \circ \tr_{c_i} \circ f_i)^* \PP^\times$ gives rise to an element of the kernel 
$\ker (\Pic (\U) \lra \Pic (C_K)).$
Note that we have an isomorphism of line bundles 
\begin{equation}
    (\id, hm\cdot \circ \tr_{c_i} \circ f_i)^*\PP  \simeq ((\id, m\cdot \circ \tr_{c_i} \circ f_i)^* \PP)^{\otimes h},
\end{equation}
obtained by using (\ref{isom:cube}). By Lemma \ref{lem:PicU} below, we conclude that $(\id, hm\cdot \circ \tr_{c_i} \circ f_i)^* \PP^\times$ becomes a trivial $\G_{m, \U}$-torsor when restricted to $\U$. It follows that $\TT$ pulls back to the trivial $\G_{m, \U}^{\rho-1}$-torsor over $\U$. In particular, the map $j_b\vert_{\U}$ admits a lift to $\TT$, which is unique up to $\G_m^{\rho-1}(\U)$. The latter is equal to $(\oh_\U(\U)^\times)^{\rho-1}=(\oh_K^{\times})^{\rho-1}$ by Lemma \ref{lem:PicU} below. 
\end{proof}
 
The following lemma is used in the proof above. 

\begin{lemma}\label{lem:PicU}
Let $\U$ be an open subscheme of $\CC^{\sm}$ as in Definition \ref{cons:U}. Then $\oh_{\U}(\U)=\oh_K$ and the kernel of the restriction $\ker(\Pic(\U)\lra  \Pic(C_K))$ is $h$-torsion. In other words, given a line bundle $L$ over $\U$, which is trivial over the generic fiber $C_K$, its $h$-th power $L^{\otimes h}$ is trivial over $\U$. 
\end{lemma}

\begin{proof}
By construction, $\U$ is regular and thus locally factorial. We do therefore not distinguish between isomorphism classes of line bundles and Weil divisors. Let $D$ be a vertical divisor on $\U$, i.e., a divisor that does not intersect the generic fiber $C_K$. We claim that $h D = 0$ in $\Pic (\U)$. As every irreducible vertical divisor on $\U$ is of the form $\U_\mathfrak{p}$ for some prime $\mathfrak{p}$ of $\oh_K$,  we may write $h D$ as 
$\sum_{\mathfrak{p}} h n_\mathfrak{p} U_{\F_\mathfrak{p}},$ where $n_{\mathfrak{p}}=0$ for almost all $\mathfrak{p}$. Clearly, $hD$ is the image of the divisor $\sum_{\mathfrak{p}} h n_{\mathfrak p}\mathfrak{p}$ along the natural map $\Pic (\oh_K) \lra \Pic (\U)$, which is trivial since $\Pic (\oh_K)$ has size $h$. 

Now let $D$ be a general element of $\Pic(\U)$, viewed as a Weil divisor on $\U$, that lies in the kernel $\ker (\Pic (\U) \lra \Pic (C_K))$. In other words, the restriction of $D$ to $C_K$ is a principal divisor $D_K = \textup{div} (f)$ for some $f$ in the function field of $C_K$. Then $\textup{div} (f)$ extends to a principal divisor on $\U$, which differs from $D$ only by a vertical divisor. The lemma thus follows.  
\end{proof}

\begin{remark}
When $h = 1$, the lemma simply says that the restriction $\Pic (\U) \lra \Pic (C_K) $ is injective. This map is of course not in general injective when $h\neq 1$.  Indeed, in this case it suffices to take $D=U_{\mathfrak{p}}\in \Div(\U)$, where $\mathfrak{p}$ is a non-principal prime ideal of $\oh_K$.  
\end{remark}


\section{The main theorem}\label{s:method}

In this section, we state a precise version of the main theoretical result of the article. We give a detailed description of the strategy of the geometric quadratic Chabauty method. 

\begin{assumption}\label{passumption} 
 Throughout, we make the following assumptions on the prime $p$. 
\begin{itemize}
\item[(a)] The curve $C_K$ has good reduction at each prime $\mathfrak{p}_1, \ldots, \mathfrak{p}_s$ of $K$ that lies above $p$. 
\item[(b)] Each $\mathfrak{p}_i$ satisfies $e(\mathfrak{p}_i/p) < p-1$. 
\item[(c)] The prime $p$ does not divide $\vert \oh_{K, \mathrm{tors}}^\times \vert$.
\end{itemize}
\end{assumption}
\noindent Note that the first condition is equivalent to requiring that $\mathfrak{p}_i \nmid \mathfrak{n}$ for each $i\in \{ 1, \ldots, s \}$, 
and that Assumption \ref{passumption} excludes only finitely many primes.

\begin{remark}
We comment on the use and necessity of Assumption \ref{passumption}. Condition $(a)$ on good reduction is used in the beginning of Section \ref{s:end} to argue that the tangent map of the lifted Abel--Jacobi map on $\mathfrak{p}$-adic residue disks is injective. It is worth mentioning that the idea of using primes of bad reduction in the Chabauty--Coleman method has been explored in \cite{stoll, KZB, KRZB}. Condition $(b)$ on the ramification degree not being too large is used to deduce that the kernel of reduction $\J(\oh_K)_0$ is a free $\Z$-module (see Section \ref{s:strategy}), and to prove the integrality of certain power series defining formal exponential and logarithm maps (see the proof of Proposition \ref{logexp}). The condition excludes the case $p=2$ even though this case can be particularly interesting, see for instance \cite{poonenstoll}. Similarly, the case $p=2$ is excluded in \cite{EL19}, as well as in the geometric linear Chabauty method \cite{spelierhashimoto}. Condition $(c)$ is used in the proof of Proposition \ref{prop:modify-E-by-local-units2} below. 
\end{remark}


\begin{notation} \label{OKp-OKP-points}
\hfill
\begin{itemize}
    \item Let $\oh_{K, p}:= \oh_{K}\otimes \Z_p$ be the $p$-adic completion of $\oh_K$. 
    It is isomorphic to the product of the $\mathfrak{p}_i$-adic completions $ \oh_{K,\mathfrak{p}_1} \times \ldots \times \oh_{K,\mathfrak{p}_s}.$
    \item Let $\overline{\oh_{K,p}}:=(\oh_K \otimes \mathbb{F}_p)_{\mathrm{red}}\simeq{ \mathbb{F}_{\mathfrak{p}_1}\times \ldots \times \mathbb{F}_{\mathfrak{p}_s}}.$
    \item   For any $\oh_K$-scheme $X$, we have natural identifications
    \[
    X(\oh_{K, p}) =\prod_{i=1}^s X_{\oh_{K,\mathfrak{p}_i}}(\oh_{K,\mathfrak{p}_i}) \qquad \text{ and } \qquad
    X(\overline{\oh_{K, p}}) = \prod_{i=1}^s  X_{\F_{\mathfrak{p}_i}}(\F_{\mathfrak{p}_i}).
    \]
    We denote the natural reduction map by 
    $$ \textup{red}: X(\oh_{K, p})\lra X(\overline{\oh_{K, p}}).$$  
   \item     Given a point $x\in X(\overline{\oh_{K, p}})$, we denote by $X(\oh_{K, p})_x$ the set $\textup{red}^{-1} (x)$, 
  namely the residue disk in $X(\oh_{K, p})$ of points that reduces to the point $x$. Likewise, we denote by $X(\oh_K)_x$ the pre-image of $X(\oh_{K, p})_x$ under the natural inclusion 
  $ X(\oh_K) \longhookrightarrow X(\oh_{K, p}),$
  which consists of rational points in the residue disk $X(\oh_{K, p})_x$. 
      \end{itemize}
 \end{notation}
   
 \begin{remark} \label{remark: all_primes}
The reason for working with all primes above $p$ simultaneously, as opposed to fixing a single prime, is explained in Section \ref{rem:singleprime}. 
 \end{remark}

\subsection{Revisiting the strategy}
Let $\U$ be an open subscheme of $\CC^{\sm}$ as in Definition \ref{cons:U}. Let $u$ be an element in the finite set $\U(\overline{\oh_{K, p}})$, and let
$t:=\tilde{j}_b^{U}(u)\in \TT(\overline{\oh_{K, p}})$
be its image in $\TT$ under the lift $\tilde{j}_b^U : \U \lra \TT$ of Proposition \ref{prop:Ulift}. Note that $\CC^{\sm}(\oh_K)$ is the disjoint union of $\U(\oh_K)$ for the finitely many choices of $\U$'s by Remark \ref{remark:U}, and each $\U(\oh_K)$ is the disjoint union of finitely many residue disks $\U(\oh_K)_u$. Thus, for the purposes of this work, it suffices to bound the size of $\U(\oh_K)_u$ for each $\U$ and each point $u \in \U(\overline{\oh_{K, p}})$.  

The key idea of the approach can be represented using the commutative diagram
\begin{equation}\label{diag:chab}
\begin{tikzcd}[row sep = 2em]
   \U(\oh_K)_u \arrow[hook]{d}{\tilde{j}_b^U}  
   \arrow[hook]{rr}
&&       \U(\oh_{K, p})_u \arrow[hook]{d}{\tilde{j}_b^U} 
 \\
   \TT(\oh_K)_t \arrow[hook]{r} 
    &  \Y_t \arrow[hook]{r} 
    & \TT(\oh_{K, p})_t,
\end{tikzcd}
\end{equation}
where the top horizontal arrow is induced by the inclusion $\oh_K \hookrightarrow \oh_{K, p}$, and 
$\Y_t := \overline{\TT(\oh_K)_t}^p$ denotes the $p$-adic closure of $\TT(\oh_K)_t$ in $\TT(\oh_{K, p})_t$. We view $\U(\oh_K)_u$ and $\U(\oh_{K, p})_u$ as a subsets of $\TT(\oh_K)_t$ and $\TT(\oh_{K, p})_t$ respectively via the map $\tilde{j}_b^U$. In particular, we have inclusions
$\U(\oh_K)_u \hookrightarrow   \U(\oh_{K, p})_u \cap \Y_t. $
As explained in the introduction,  the goal is to bound the intersection 
    \begin{equation}\label{intersection}
     \U(\oh_{K, p})_u \cap \Y_t,
    \end{equation}
    which takes place in the $p$-adic analytic manifold $\TT(\oh_{K, p})_t$.

\begin{remark}
    For this intersection to have a chance to be finite, some condition must be imposed in the style of the original Chabauty condition  $r<g$. We will come back to this point in Section \ref{s:chabcondition} after stating the main technical result of the paper. 
\end{remark}

\subsection{The key technical result}\label{s:strategy} 

We give a description of $\Y_t$, which is a crucial ingredient in bounding the intersection (\ref{intersection}). 
\begin{notation} \label{notation:main_thm_sec} 
\hfill
\begin{itemize} 
\item Recall that $r:=\rank_\Z J_K(K)$ is the Mordell--Weil rank of $J_K$ over $K$. 
\item  
Let $\J(\oh_K)_0$ denote the subgroup of $J_K(K) = \J(\oh_K)$ given by kernel 
$$ \J(\oh_K)_0 := \ker \big( \textup{red}: \J(\oh_K) \lra \J(\overline{\oh_{K, p}}) \big).$$ 
\item 
Let $q^*$ denote the exponent of $\G_m(\overline{\oh_{K, p}})$, i.e., the least common multiple of $q_i-1=\# \mathbb{F}_{\mathfrak{p}_i}-1$ for $i\in \{ 1, \ldots, s \}$.
\item 
For each $i\in \{ 1, \ldots, s \}$, let $k_i = k_{\mathfrak{p}_i}=e_{\mathfrak{p}_i}f_{\mathfrak{p}_i}$ be the $\Z_p$-rank of $\oh_{K, \mathfrak{p}_i}$. Note that the rank of $\oh_{K, p}$ as a $\Z_p$-module is $\sum_{\mathfrak{p_i}\vert p} k_i =d$, where $d$ is the degree of $K$ over $\Q$. 
\end{itemize}
\end{notation}

 By Assumption \ref{passumption} $(b)$ on $p$, for each $i\in \{ 1, \ldots, s \}$, the reduction 
$\J(\oh_K)\lra \J(\F_{\mathfrak{p}_i})$ is injective on the torsion points of $\J(\oh_K)$ by \cite[Appendix]{katz1980galois}. Hence $\J(\oh_K)_0$ is a free $\Z$-module of rank $r$. We fix a basis $\mathbf{x}=\{ x_1, \ldots, x_r \}$. 

The idea is to parametrize the $p$-adic closure $\Y_t = \overline{\TT(\oh_K)_t}^p$ using the set
$$
(\G_m^{\rho-1}(\oh_K)_{\tf} \times \J(\oh_K)_{j_b(u)})\otimes \Z_p,
$$
where the subscript ``$\tf$'' stands for the torsion free quotient viewed as a free subgroup of $\G_m^{\rho-1}(\oh_K)$ via a fixed lift $\G_m^{\rho-1}(\oh_K)_{\tf}\lra \G_m^{\rho-1}(\oh_K)$.
In Section \ref{s:construct} (see Definition \ref{def:mapEprime}), we will define a map 
\[
E' : \G_m^{\rho-1}(\oh_K)_{\tf} \times \J(\oh_K)_{j_b(u)} \lra \TT(\oh_{K,p})_t,
\]
which depends on an initial choice of finitely many points of $\PP^{\times, \rho-1}(\oh_K)$, a choice which itself depends on the basis $\mathbf{x}$ fixed above. The map $E'$ enjoys the following properties (see Propositions \ref{prop:modify-E-by-local-units2} and \ref{mapE'}):


\begin{proposition}\label{prop:pinterpolation}
Upon fixing a basis for the free $\Z$-module $\G_m^{\rho-1}(\oh_K)_{\tf} \subset \G_m^{\rho-1}(\oh_K)$, and identifying the set $\G_m^{\rho-1}(\oh_K)_{\tf}\times \J(\oh_K)_{j_b(u)}$ with $\Z^{\delta(\rho-1)+r}$ using the basis $\mathbf{x}$ above, 
the map $E'$ can be viewed as a map $\Z^{\delta(\rho-1)+r} \lra \TT(\oh_{K, p})_t$. With respect to these choices of bases, $E'$ admits an explicit description in terms of the partial composition laws of Section \ref{S:PoincareTorsor}, and satisfies 
\begin{equation}\label{eq:Eprop}
    E'(q^* \Z^{\delta(\rho-1)+r})\subset \TT(\oh_K)_t \subset E'(\Z^{\delta(\rho-1)+r}),
\end{equation}
where $q^*$ is the integer defined in Notation \ref{notation:main_thm_sec}. 
\end{proposition}

From now on, we fix a basis $\mathbf{u}=\{ u_1, \ldots, u_{\delta} \}$ of $\G_m^{\rho-1}(\oh_K)_{\tf} \subset \G_m^{\rho-1}(\oh_K)$ and view $E'$ as a map $\Z^{\delta(\rho-1)+r} \lra \TT(\oh_{K, p})_t$. The domain of this map is a dense subspace of $\Z_p^{\delta(\rho-1)+r}$ endowed with the $p$-adic topology. On the other hand, the target space $\TT(\oh_{K, p})_t$ naturally carries a $p$-adic topology. 
In Section \ref{s:interpolation}, we ``$p$-adically interpolate'' the map $E'$ to get the following  result:

\begin{theorem}\label{thm:pinterpolation}
With respect to the fixed bases $\mathbf{x}$ and $\mathbf{u}$ above, there is a unique continuous map $\kappa=\kappa_{\mathbf{x},\mathbf{u}}$ 
making the diagram
\[
\begin{tikzcd}[column sep = 1.2em]
\Z^{\delta(\rho-1)+r}
 \arrow{rrr}{E'} \arrow[hook]{d}
&&& \TT(\oh_{K, p})_t \\
\Z_p^{\delta(\rho-1)+r} \arrow[dashed]{urrr}[swap]{\exists ! \kappa} 
&&&
\end{tikzcd}
\]
commute.
We call the map $\kappa$ the $p$-adic interpolation of $E'$ (with respect to the bases $\mathbf{x}$ and $\mathbf{u}$).
Moreover, the choice of a regular system of parameters $\mathbf{t}_{\mathfrak{p}}$ for $\TT\lra \J$ at $t_\mathfrak{p}\in \TT(\F_\mathfrak{p})$ for each $\mathfrak{p}$, as well as an isomorphism of $\Z_p$-modules $\oh_{K, \mathfrak{p}}\simeq \Z_p^{k_{\mathfrak{p}}}$, yields a homeomorphism 
$\tilde{\mathbf{t}} : \TT(\oh_{K, \mathfrak{p}})_t\simeq \Z_p^{(g+\rho-1)k_{\mathfrak{p}}}$ and uniquely determines a $(g+\rho-1)d$-tuple of convergent power series $\kappa_i\in \Z_p\langle z_1, \ldots, z_{\delta(\rho-1)+r} \rangle$ such that
\[
\tilde{\mathbf{t}}\circ \kappa = (\kappa_1, \ldots, \kappa_{(g+\rho-1)d}) : \Z_p^{\delta(\rho-1)+r} \lra \Z_p^{(g+\rho-1)d}.
\]
Thus, $\kappa$ is a map of $p$-adic analytic manifolds.
\end{theorem}

\begin{remark}\label{rem:dependencies}
A lot of dependencies have been suppressed in the above notations. The map $\kappa$ of Theorem \ref{thm:pinterpolation} depends on the point $t\in \TT(\overline{\oh_{K,p}})$ and a choice of lift $\tilde{t}$ in $\TT(\oh_{K})$, the choice of basis $\mathbf{x}$ for the free $\Z$-module $\J(\oh_K)_0$, the finitely many choices of initial points of $\P^{\rho-1}(\oh_K)$ needed to define the map $E'$, and finally the basis $\mathbf{u}$ for the free $\Z$-module $\G_m(\oh_K)_{\tf}\subset \G_m(\oh_K)$. Note however, that $\kappa$ does not depend on the choice of local parameters for $\TT(\oh_{K,p})_t$, even though its expression in terms of convergent $p$-adic power series does.
\end{remark}

\begin{corollary}\label{coro:imagekappa}
The image of the map $\kappa$ is the $p$-adic closure $\Y_t = \overline{\TT(\oh_K)_t}^p$. 
\end{corollary}

\begin{proof}
Since $\Z_p^{\delta(\rho-1)+r}$ is compact and $\kappa$ is continuous, the image of $\kappa$ is closed in $\TT(\oh_{K, p})_t$. Since $\kappa$ extends $E'$, the second containment of (\ref{eq:Eprop}) implies that $\im \kappa$ contains $\TT(\oh_K)_t$, and thus also contains $\Y_t$. On the other hand $q^*\Z^{\delta(\rho-1)+r}$ is dense in $\Z_p^{\delta(\rho-1)+r}$ since $q^*$ is coprime to $p$. By continuity of $\kappa$, we have 
\[
\im \kappa = E'\left(\overline{q^*\Z^{\delta(\rho-1)+r}}\right) \subset \overline{E'(q^*\Z^{\delta(\rho-1)+r})}\subset \Y_t = \overline{\TT(\oh_K)_t}^p,
\]
where the last containment uses the first inclusion of (\ref{eq:Eprop}). This concludes the proof.
\end{proof}

In Section \ref{s:end}, we use the map $\kappa$ to pull-back the $p$-adically convergent power series cutting out $\U(\oh_{K,p})_u$ inside $T(\oh_{K,p})_t$. These pull-backs belong to $R:=\Z_p \langle z_1, ..., z_{\delta(\rho - 1)+ r} \rangle$ and generate an ideal $I_{\U,u}=I_{\U,u,\kappa}$ whose precise definition is given in Definition \ref{AI}.
The more precise form of Theorem \ref{thm:intro_main} is then the following, which we prove in Section \ref{s:end}:

\begin{theorem} \label{thm:main_rough_form}
If $\overline{A}_{\U, u} := \big( R/I_{\U, u} \big) \otimes \F_p $ is finite dimensional over $\F_p$, then the number of rational points in $\U (\oh_K)_u$ is finite and bounded by 
$$ | \U (\oh_K)_u | \le \dim_{\F_p} \overline{A}_{\U, u}. $$
\end{theorem}
As discussed in the introduction, we expect this bound to prove useful in determining the rational points of curves in many new examples.

\subsection{Chabauty conditions}\label{s:chabcondition}

We end this section with a discussion of the Chabauty condition. 
We retain all notations and assumptions from the previous sections, in particular Assumption \ref{passumption} on the prime $p$. 

 For each prime $\mathfrak{p}$ above $p$, the scheme $\TT\times_{\oh_K} \spec \oh_{K, \mathfrak{p}}$ is smooth over $\oh_{K, \mathfrak{p}}$ of relative dimension $g+\rho-1$ and the set $\TT(\oh_{K, \mathfrak{p}})$ is equipped with the structure of a $p$-adic analytic manifold of dimension $(g+\rho-1)k_\fp$. In particular, $\TT(\oh_{K, p})$ is a $p$-adic analytic manifold of dimension
$$(g+\rho-1)\sum_{\fp\vert p} k_\fp=(g+\rho-1)d.$$ 
By Theorem \ref{thm:pinterpolation} and Corollary \ref{coro:imagekappa}, the $p$-adic analytic submanifold $\Y_t = \overline{\TT(\oh_K)_t}^p$ is parametrized by $\Z_p^{\delta(\rho-1)+r}$ via the map of $p$-adic analytic manifolds $\kappa : \Z_p^{\delta(\rho-1)+r} \twoheadrightarrow \Y_t$.
The dimension  of the  $p$-adic analytic manifold $\Y_t$ is therefore bounded above by 
$$\dim \Y_t \leq \delta(\rho-1)+r.$$
Finally, we observe that $\U(\oh_{K, p})$ has dimension $d$ as a $p$-adic analytic manifold.

A necessary condition for the intersection $\U(\oh_{K, p})_u \cap \Y_t$ in (\ref{intersection}) to be finite is the inequality on dimensions of $p$-adic analytic manifolds 
$$
\codim \U(\oh_{K, p}) + \codim \Y_t \geq \dim \TT(\oh_{K, p}), 
$$
where the codimensions are taken with respect to the ambient manifold $\TT(\oh_{K, p})$.
This inequality is satisfied if we require the weaker inequality 
$$\delta(\rho-1)+r \leq (g+\rho-2)d,$$ which in turn is equivalent to the condition
\begin{equation}\label{eq:chabbound}
    r\leq (g-1)d+(\rho-1)(r_2+1).
\end{equation}

\begin{definition} \label{definition:chabcondition}
We say that a smooth, projective and geometrically connected curve $C_K$ of genus $g\geq 2$ over a number field $K$ satisfies the geometric quadratic Chabauty condition if 
the inequality (\ref{eq:chabbound}) holds. 
\end{definition}

\begin{remark} \label{remark:chabcondition}
The term ``geometric'' distinguishes condition (\ref{eq:chabbound}) from other Chabauty type conditions associated to the various methods discussed in Sections \ref{s:Qmethods} and \ref{s:intro:GQCNF}. We briefly compare these conditions: 

    \begin{itemize}
        \item When $K=\Q,$ the condition (\ref{eq:chabbound}) becomes $r\leq g+\rho-2,$ which is the same condition as in the geometric quadratic Chabauty method over $\Q$ of Edixhoven and Lido \cite{EL19} discussed in Section \ref{s:intro:GQC}.
        \item Siksek \cite{siksek} extended the classical Chabauty--Coleman method to arbitrary number fields using Weil restrictions. This is the Restriction of Scalars (RoS) Chabauty method of Section \ref{s:intro:RoSC}. The method is expected to be successful when 
        $
        r\leq (g-1)d
        $
       and certain additional conditions are assumed (see the discussion in Sections \ref{s:intro:RoSC} and \ref{s:intro:RoSQC}).
        Under such conditions, the geometric quadratic Chabauty method is thus expected to go beyond the RoS Chabauty method.
        \item In their recent work \cite{BBBM19}, Balakrishnan, Besser, Bianchi, and M\"uller extended the method of quadratic Chabauty to number fields in the case of hyperelliptic or bielliptic curves. In Section \ref{s:intro:RoSQC}, we referred to this method as the effective RoS quadratic Chabauty method. It performs under the relaxed condition
        $ r \le (g-1) d + r_2 + 1. $
        The geometric Chabauty condition agrees with this when $\rho$ is equal to $2$, and in fact generalizes this bound for $\rho \ge 2$. 
        \item Dogra \cite{Dogra19} has recently proved that, under the extra condition \eqref{hom} on $J_K$ and $K$, a certain ``arithmetic quadratic Chabauty condition'' implies that the set $C_K(K_\mathfrak{p})_2$ appearing in the RoS Chabauty--Kim method of Section \ref{s:intro:RoSQC} is finite for some $\mathfrak{p}$ above the split prime $p$. 
 If one assumes the finiteness of the $p$-primary part of the Shafarevich--Tate group for $J_K$, then the aforementioned Chabauty condition of Dogra is (a slightly relaxed version of) the geometric condition (\ref{eq:chabbound}). We refer to \cite[Proposition 1.1 \& Remark 1.5]{Dogra19} for further details. 
    \end{itemize}

\end{remark}

\section{The parametrization of $\Y_t$}\label{s:closure}

We retain the notations of Section \ref{s:method}.  The goal of this section is to prove Theorem \ref{thm:pinterpolation}, i.e., to describe the $p$-adic closure $\Y_t$ of $\TT(\oh_K)_t$ inside $\TT(\oh_{K, p})_t$. 

\subsection{Construction of the map $E'$}\label{s:construct} 

We construct the map $E'$ of Proposition \ref{prop:pinterpolation}. This map, and subsequently its $p$-adic interpolation $\kappa$ of Theorem \ref{thm:pinterpolation}, depends on the choice of $r^2+2r$ initial points in $\PP^{\times, \rho-1}(\oh_K)$. The definition of $E'$ is achieved in $3$ steps.

\begin{notation} \label{NotationForE}
\hfill
\begin{itemize}
\item Fix a basis $\mathbf{x}=\{ x_1, \dots, x_r \}$ of the free $\Z$-module $\J(\oh_K)_0 = \ker \big(\J(\oh_K) \lra \J(\overline{\oh_{K, p}}) \big)$. 
\item Recall that $u$ is a fixed $\overline{\oh_{K, p}}$-point of $\U$ and $t=\tilde{j}_b^U(u)$. 
Denote by $\tilde{t}$ any lift of $t$ to an $\oh_K$-point of the torsor $\TT$. Such a lift is assumed to exist, as otherwise $\U(\oh_K)_u = \emptyset$, and we are done. Denote by $x_{\tilde{t}}$ its image in $\J(\oh_K)$. 
\item Let $\TT(\oh_K)_{j_b(u)}$ be the set of points of $\TT(\oh_K)$ whose image in $\J(\overline{\oh_{K, p}})$ is $j_b(u)$.
\end{itemize}
For the reader's convenience, the points and sets are pictured in the following diagrams: 
\[
\begin{tikzcd} 
 & t  \arrow[d, mapsto] & \tilde t  \arrow[l, squiggly, mapsto, swap, "\textup{red}"] \arrow[d, mapsto] \\ 
u \arrow[r, mapsto, "j_b"] \arrow[ru, mapsto, "\tilde j_b^U"]  & j_b (u) & 
 x_{\tilde t} \arrow[l, squiggly, mapsto, swap, "\textup{red}"] 
\end{tikzcd}
\quad  \quad
\begin{tikzcd}[column sep = 1em]
 & \TT (\oh_K)_t \subset \TT(\oh_K)_{j_b(u)} \arrow[d] \\ 
\U(\oh_K)_u \arrow[r, hook, "j_b"] \arrow[ru, hook, "\tilde j_b^U"]  & \J(\oh_K)_{j_b(u)}.
\end{tikzcd}
\] 
\end{notation}

\subsubsection{Construction of the map $D$}\label{ss:D}

The first step is the construction of a section 
\[ 
D : \J(\oh_K)_{j_b(u)} \lra \TT(\oh_K)_{j_b(u)},
\]
essentially by choosing a lift for each point $x\in \J(\oh_K)_{j_b(u)}$.  
This can be done in a coherent way using the biextension laws (important for the $p$-adic interpolation below), once we choose finitely many initial points
$P_{i, j}, R_{i}, S_{j} \in \PP^{\times, \rho-1}(\oh_K)$, $0\leq i,j\leq r$, lifting the following points of 
$\J \times (\Jo)^{\rho-1}(\oh_K)$:
\begin{align*}
P_{i, j} &\longmapsto \Big(x_i,\, \underline{f}(hmx_j)\Big)=\Big(x_i,\, hm\underline{f}(x_j)\Big),\\
R_{i} &\longmapsto \Big(x_i,\, ({hm\cdot}\circ \tr_{\underline{c}}\circ \underline{f})(x_{\tilde{t}})\Big),\\
S_{j} &\longmapsto \Big(x_{\tilde{t}},\, \underline{f}(hmx_j)\Big)=\Big(x_{\tilde{t}},\, hm\underline{f}(x_j)\Big).
\end{align*}
Here $\underline{f}$ is given by the functions $f_i$ of Notation \ref{notation:f_i}. 
Note that the points to be lifted are of the form $(*, h\cdot *)$, and the existence of such lifts is guaranteed by Lemma \ref{lem:h}. Note also that unlike the situation of \cite{EL19}, these lifts are no longer defined up to a finite choice, as they are now parametrized by $\G_m^{\rho-1}(\oh_K)$. 

As a set, $\J(\oh_K)_{j_b(u)}=x_{\tilde{t}}+\J(\oh_K)_0=x_{\tilde{t}}+\bigoplus_{i=1}^r \Z x_i$ is in bijection with $\Z^r$: any point of $\J(\oh_K)_{j_b(u)}$ can be expressed as $x_{\underline{n}}:=x_{\tilde{t}}+\sum_{i=1}^r n_i x_i$ for a unique $\underline{n} \in \Z^{r}$.
Given $\underline{n} \in \Z^{r}$, we will now define a lift $D(\underline{n}):=D(x_{\underline{n}})$ of $x_{\underline{n}}$ using the above initial points and the biextension laws. Define 
\[
\begin{array}{lcl}
A(\underline{n}):={\sum}_{2, j}n_j \cdot_2 S_j &\longmapsto & \Big(x_{\tilde{t}},\, hm\underline{f}\big(\sum_{i}n_ix_i\big)\Big),\\
B(\underline{n}):={\sum}_{1, i}n_i \cdot_1 R_i &\longmapsto & \Big(\sum_i n_i x_i,\, ({hm\cdot}\circ \tr_{\underline{c}}\circ \underline{f})(x_{\tilde{t}})\Big),\\
C(\underline{n}):={\sum}_{1, i}n_i \cdot_1 \left( {\sum}_{2, j}n_j \cdot_2 P_{i, j} \right) &\longmapsto & \Big(\sum_i n_i x_i,\, hm\underline{f}\big(\sum_{i}n_ix_i\big)\Big).
\end{array}
\]
Here $\cdot_1$ and $\cdot_2$ denote the iterations of the operations $+_1$ and $+_2$ respectively, and similarly for $\sum_1$ and $\sum_2$.
We then define
$$D(\underline{n}):=\big(C(\underline{n})+_2B(\underline{n})\big)+_1\big(A(\underline{n})+_2\tilde{t}\big),$$
lying over the point
$$\big(x_{\underline{n}}, \alpha(x_{\underline{n}})\big):=\left(x_{\tilde{t}} + \sum_{i}n_ix_i, \big({hm\cdot} \circ \tr_{\underline{c}}\circ \underline{f}\big)\left(x_{\tilde{t}}+ \sum_in_i x_i\right) \right)$$
in $\J\times (\Jo)^{\rho-1}(\oh_K)$. To see this, note that the point $\tilde{t} \in \TT (\oh_K)$, when viewed as an point in $\PP^{\times, \rho-1}$, lies over the point 
$(x_{\tilde{t}}, ({hm\cdot}\circ \tr_{\underline{c}}\circ \underline{f})(x_{\tilde{t}}) ).$ In particular, $D(\underline{n})$ belongs to $\TT(\oh_K)$ and lies above $x_{\underline{n}}$ in $J(\oh_K)_{j_b(u)}$.

\subsubsection{Construction of the map $E$}\label{ss:E}

The second step is one of the main technical innovations of this work compared to \cite{EL19}: we extend
$D: \J(\oh_K)_{j_b(u)} \lra \TT(\oh_K)_{j_b(u)}$ to a map 
$$
E: \G_m(\oh_K)_{\mathrm{tf}}^{\rho-1} \times \J(\oh_K)_{j_b(u)} \lra  \TT(\oh_K)_{j_b(u)}, \qquad E(\zeta, x):=\zeta\cdot D(x).
$$
The subscript $\mathrm{tf}$ stands for ``torsion-free quotient'' as before, viewed as a subgroup of $\G_m(\oh_K)^{\rho-1}$ via a map determined by an arbitrary but fixed choice of a splitting $\oh_{K, \mathrm{tf}}^{\times} \lra \oh_K^{\times}$.
It will be important later on that this map admits an expression in terms of $+_1$, $+_2$ and their iterates $\cdot_1, \cdot_2$. 

\begin{notation}\label{not:VKL}
\hfill
\begin{itemize}
\item 
Fix a free basis $\mathbf{u}=\{u_1, \dots, u_\delta\}$ of  $\oh_{K,\mathrm{tf}}^{\times}=\G_m(\oh_K)_{\mathrm{tf}}$, viewed as a subgroup of $\oh_K^{\times}$ via 
the same choice of splitting as above.
\item For each $(\rho-1)$-tuple $u_{k, l}=(1, \dots, 1, u_{k}, 1, \dots, 1) \in\G_m^{\rho-1}(\oh_K)$, where $u_k$ sits in the $l$-th coordinate, we denote the corresponding elements in $\PP^{\times}_{\mid_{\J \times 0}}(\oh_K)$ above the point $(x_{\tilde{t}}, 0)$ by $V_{k, l}$, in the sense of Formula (\ref{rem:internal-Gm-action}) but with $\PP^\times$ in place of $P_K^\times$. Likewise we denote the corresponding element above $(x_i, 0)$ by $W_{k, l, i}$.
\end{itemize}
\end{notation}

\begin{definition} \label{map:E}
For $\underline{n} \in \Z^r$, $k\in \{ 1, \ldots, \delta \}$ and $l\in \{ 1, \ldots, \rho-1 \}$, we define the element 
$$U_{k, l}(\underline{n}):=V_{k, l}+_1 \sum_{1, i}n_i\cdot_1 W_{k, l, i},$$ 
which represents multiplication by $u_{k, l}$ and lies above the point 
$(x_{\underline{n}}, 0).$
\end{definition}

Any element of $\G_m(\oh_K)_{\mathrm{tf}}^{\rho-1}$ can be described as $\zeta_{\underline{m}}:=(\prod_{k=1}^\delta u_k^{m_{k,l}})_{l=1}^{\rho-1}$ for a unique $(\rho-1)$-tuple of $\delta$-tuples of integers $\underline{m}=(m_{k, l})_{\substack{1\leq k \leq \delta \\ 1\leq l \leq \rho-1}} \in \Z^{\delta(\rho-1)}.$
Given $\underline{m}=(m_{k, l})_{\substack{1\leq k \leq \delta \\ 1\leq l \leq \rho-1}} \in \Z^{\delta(\rho-1)},$ the map $E$ is described explicitly by the following formula:
\begin{equation}\label{explicitE}
E(\underline{m}, \underline{n}):=E(\zeta_{\underline{m}},x_{\underline{n}})=\left({\sum}_{2, k, l}m_{k, l}\cdot_2U_{k, l}(\underline{n})\right)+_2 D(\underline{n})\in \TT(\oh_K).
\end{equation}

One easily checks that $E(\underline{m}, \underline{n})$ lies over the same point 
$\big(x_{\underline{n}}, \alpha(x_{\underline{n}})\big) \in \J\times \Jo(\oh_K)$ as does $D(\underline{n})$. After all, the parameters $\underline{m}$ just encode part of the $\G_m^{\rho-1}$-action on the fibers, as was previously indicated. Passing from $\oh_K$ to $\overline{\oh_{K, p}},$ the contribution of the $x_i$'s vanishes and the point becomes 
$(j_b(u), ({hm\cdot}\circ \tr_{\underline{c}}\circ \underline{f})(j_b(u))).$ In other words, we have 
$E(\underline{m}, \underline{n}) \in \TT(\oh_K)_{j_b(u)}.$

\begin{proposition}\label{prop:bijective-E}
The map
$\oh_{K, \tors}^{\times, \rho-1}\times \Z^{\delta(\rho-1)+r} \lra \TT(\oh_K)_{j_b(u)}$ defined by mapping
$(\varepsilon, \underline{m}, \underline{n})$ to $\varepsilon \cdot E(\underline{m}, \underline{n})$
is bijective.
\end{proposition}
\begin{proof}
This is immediate after tracking the definitions. As $\underline{n} \in \Z^r$ varies, $x_{\underline{n}}=x_{\tilde {t}}+\sum_i n_ix_i$ runs over all the points of $\J(\oh_K)$ that reduce to $j_b(u)$, and $D(\underline{n})$ provides a single  point in $\TT(\oh_K)_{j_b(u)}$ lying above $x_{\underline{n}}$. In particular, the map $D$ is injective. To get all the points of $\TT(\oh_K)_{j_b(u)}$, one needs to move these around by the simply transitive $\G_m^{\rho-1}(\oh_K)$-action. Since $E(\underline{m}, \underline{n})=\zeta_{\underline{m}}\cdot D(\underline{n})$ accounts for the torsion-free part of the action by the above discussion, what is left is the torsion part, hence the factor $\oh_{K, \tors}^{\times, \rho-1}$.
\end{proof}

\subsubsection{Construction of the map $E'$}\label{ss:Eprime}

For the purpose of computing rational points, we wish to parametrize $\TT(\oh_K)_t$ instead of all of $\TT(\oh_K)_{j_b(u)}$. In this section, we modify the choices of initial points in the above construction of the map $E$ to obtain a map $E'$ that has the advantage that it lands in the correct residue disk, i.e., such that $E'(\underline{m}, \underline{n})$ reduces to $t$ in $\TT(\overline{\oh_{K, p}})$ for all $(\underline{m}, \underline{n}) \in \Z^{\delta(\rho-1)+r}$. 

The starting point is the following observation, which asserts that the desired property is already satisfied by $E$ on a certain finite index subgroup of $\Z^{\delta(\rho-1)+r}$. 

\begin{proposition}\label{prop:q-1-kills}
Let $q^*$ be the exponent of $\G_m(\overline{\oh_{K, p}}),$ or in other words the least common multiple of $q_i-1=\# \mathbb{F}_{\mathfrak{p}_i}-1$ for $i\in \{ 1, 2, \dots, s\}$. Then 
$$E(q^*\underline{m}, q^*\underline{n})\in \TT(\oh_K)_{t}, \qquad \text{ for all } (\underline{m}, \underline{n})\in \Z^{\delta(\rho-1)+r}.$$
\end{proposition}

\begin{proof}
We need to show that $E(q^*\underline{m}, q^*\underline{n})$ reduces to the point $t$ in $\TT(\overline{\oh_{K, p}})$. To that end, we consider the elements
$$A(q^*\underline{n}),
\qquad\qquad\qquad B(q^*\underline{n}),
\qquad\qquad \qquad C(q^*\underline{n}),
\qquad\quad U_{k, l}(q^*\underline{n}),$$
lying respectively in the fibers of the $\overline{\oh_{K, p}}^{\times, \rho-1}$-torsor $\PP^{\times, \rho-1}(\overline{\oh_{K, p}})$ above the points 
$$(j_b(u), 0),
\qquad(0, ({hm\cdot}\circ \tr_{\underline{c}}\circ \underline{f})(j_b(u))), 
\qquad \quad (0, 0), 
\qquad \quad (j_b(u), 0).$$
The $\overline{\oh_{K, p}}^{\times,\rho-1}$-torsors obtained from $\PP^{\times, \rho-1}$ by taking the fibers over each of these points in $\J\times \Jo(\overline{\oh_{K, p}})$ are all trivial, as at least one coordinate is zero in each case. That is, they are groups isomorphic to $\overline{\oh_{K, p}}^{\times, \rho-1}$, whose group operation is given by $+_2$ in the cases of $A$ and the $U_{k, l}$'s, by $+_1$ in the case of $B$, and by either of the two operations in the case of $C$, since $+_1$ and $+_2$ agree above the point $(0, 0)$. By linearity of their definitions, we obtain
$$A(q^*\underline{n})=q^{*}\cdot_2 A(\underline{n})=1,  \qquad B(q^*\underline{n})=q^{*}\cdot_1 B(\underline{n})=1, \qquad U_{k, l}(q^*\underline{n})=q^*\cdot_2U_{k, l}(\underline{n})=1, $$
as elements of $\overline{\oh_{K, p}}^{\times, \rho-1}.$ Finally, for $C$ we have
$$C(q^*\underline{n})=q^*\cdot_1\left({\sum}_{1, i}n_i \cdot_1 \left({\sum}_{2, j}q^*n_j\cdot_2 P_{i, j}\right)\right)=1.$$
Putting these equations together, we obtain
$$D(q^*\underline{n})=(1+_21)+_1(1+_2t)=t.$$ 
Beware of the clash of additive and multiplicative notations here. We therefore have
$$E(q^*\underline{m}, q^*\underline{n})=q^* \cdot_2\left({\sum}_{2, k, l}m_{k, l}\cdot_2U_{k, l}(q^*\underline{n})\right)+_2 D(q^*\underline{n})=1+_2t=t\;.$$ 
This verifies the claim.\end{proof}

To get the desired map $\G_m(\oh_K)_{\mathrm{tf}}^{\rho-1} \times \J(\oh_K)_{j_b(u)}\simeq \Z^{\delta(\rho - 1)+ r} \lra \TT(\oh_K)_t$, which agrees with $E$ on $q^* \Z^{\delta(\rho - 1)+ r}$, is strictly speaking not possible. We can however still obtain a map $E'$ on the entire group $\Z^{\delta(\rho-1)+r}$ that agrees with $E$ on the subgroup $q^* \Z^{\delta(\rho - 1)+ r}$, but at the cost of allowing $p$-adic coefficient, i.e., a map 
$$
E' : \G_m(\oh_K)_{\mathrm{tf}}^{\rho-1} \times \J(\oh_K)_{j_b(u)}\simeq \Z^{\delta(\rho - 1)+ r} \lra \TT(\oh_{K,p})_t.
$$
The following proposition is the key to defining this map.

\begin{proposition}\label{adjustE}
For any $(\underline{m},\underline{n})\in \Z^{\delta(\rho-1)+r}$, there is a unique $(\rho-1)$-tuple of roots of unity of prime-to-$p$ orders $\xi(\zeta_{\underline{m}}, x_{\underline{n}})=\xi(\underline{m},\underline{n}) \in \oh_{K, p}^{\times, \rho-1}$ such that $\xi(\underline{m}, \underline{n})\cdot E(\underline{m}, \underline{n})$ belongs to $\TT(\oh_{K, p})_t$.
\end{proposition}

\begin{proof}
There is a unique multiplicative lift of units 
\begin{equation}\label{eq:liftiota}
\iota:{\overline{\oh_{K, p}}^{\times} = \mathbb{F}_{\mathfrak{p}_1}^\times \times \dots \times \mathbb{F}_{\mathfrak{p}_s}^\times \longhookrightarrow \oh_{K, \mathfrak{p}_1}^\times\times \dots \times \oh_{K, \mathfrak{p}_s}^\times=\oh_{K, p}^{\times}},
\end{equation}
which is
right inverse to the reduction map, and which maps precisely onto the prime-to-$p$ part of the roots of unity in $\oh_{K, p}$. 
Denote by $\iota$ the induced map $\G_m^{\rho-1}(\overline{\oh_{K, p}})\lra \G_m^{\rho-1}(\oh_{K, p})$ also. 
Since the action of $\G_m^{\rho-1}(\overline{\oh_{K, p}})$ on $\TT(\overline{\oh_{K, p}})_{j_b(u)}$ (i.e., the fiber of ${\TT(\overline{\oh_{K, p}})}$ containing $t$) is simply transitive, it follows that each $\iota(\G_m^{\rho-1}(\overline{\oh_{K, p}}))$-orbit of $\TT(\oh_{K, p})_{j_b(u)}$ contains a unique point of $\TT(\oh_{K, p})_t$. The assertion follows.
\end{proof}

\begin{definition}\label{def:mapEprime}
Given the above notations, we define the map 
$$
E' : \G_m(\oh_K)_{\mathrm{tf}}^{\rho-1} \times \J(\oh_K)_{j_b(u)}\simeq \Z^{\delta(\rho - 1)+ r} \lra \TT(\oh_{K,p})_t, \qquad E'(\underline{m},\underline{n}):=\xi(\underline{m}, \underline{n})\cdot E(\underline{m}, \underline{n}),
$$
where $\xi(\underline{m}, \underline{n})$ is the $(\rho-1)$-tuple of roots of unity of prime-to-$p$ orders of Proposition \ref{adjustE}.
\end{definition}

\begin{remark}
Note that the definition of $E'$ depends on the fixed basis $\mathbf{x}$ of $\J(\oh_K)_0$ and the initial choice of finitely many points of $\P^{\rho-1}(\oh_K)$ needed to define the map $D$. However, it does not depend on the choice of the basis $\mathbf{u}$ of $\G_m(\oh_K)_{\tf}$.
\end{remark}

\begin{proposition} \label{prop:modify-E-by-local-units2}
The map $E'$ satisfies the following properties:
\begin{enumerate}
\item The inclusions 
$\TT(\oh_K)_t \subseteq E'(\Z^{\delta(\rho-1)+r}) \subseteq \TT(\oh_{K, p})_t$
hold,
where $\TT(\oh_K)$ is viewed as a subset of $\TT(\oh_{K, p})$ via the canonical map.  
\item We have the equality $\xi(q^*\Z^{\delta(\rho-1)+r})=1$. In particular, the maps $E$ and $E'$ agree on the subgroup $q^*\Z^{\delta(\rho-1)+r}$. 
\end{enumerate}
\end{proposition}

\begin{proof}
Part (2) follows directly from Propositions~\ref{prop:q-1-kills} and the uniqueness of $\xi(\underline{m}, \underline{n})$ in Proposition \ref{adjustE}. Let us prove (1). Given $Q$ in $\TT(\oh_K)_t \subseteq \TT(\oh_K)_{j_b(u)}$, there is by Proposition~\ref{prop:bijective-E} a unique $\varepsilon$ in $\oh_{K, \tors}^{\times, \rho-1}$ and a unique $(\underline{m}, \underline{n})$ in $\Z^{\delta(\rho-1)+r}$ such that $\varepsilon E(\underline{m}, \underline{n})=Q$. Using the fact that $\oh_{K, \tors}^{\times}$ embeds into the prime-to-$p$ part of $\oh_{K, p, \tors}^{\times}$ by Assumption \ref{passumption} $(c)$ that $p$ does not divide $\vert \oh_{K, \tors}^\times \vert$, it follows that $\varepsilon$ may be treated as a uniquely determined element of $\oh_{K, p}^{\times, \rho-1}$, whose order is finite and coprime to $p$. By the uniqueness of $\xi(\underline{m}, \underline{n})$ in Proposition \ref{adjustE}, we have $\varepsilon=\xi(\underline{m}, \underline{n}),$ so that 
$Q=\varepsilon E(\underline{m}, \underline{n})=\xi(\underline{m}, \underline{n})E(\underline{m}, \underline{n})=E'(\underline{m}, \underline{n}).$
\end{proof} 

\subsubsection{Description of $E'$ in terms of biextension laws}

An important feature for proving a $p$-adic interpolation statement (Theorem \ref{thm:pinterpolation}) is an explicit description of $E'$ in terms of the biextension laws. The following can also be taken as an alternative construction of the map $E'$.

The strategy for describing $E'$ is to modify the choices of the initial points in the construction of the maps $D$ and $E$. Note that the images $\overline{P_{i, j}}, \overline{R_i}, \overline{S_j}$ in ${\PP^{\times, \rho-1}(\overline{\oh_{K, p}})}$ lie over points of the form $(0, *), (0, *)$ and $(*, 0)$ respectively. The fibers over these points are canonically isomorphic to $\G_m^{\rho-1}(\overline{\oh_{K, p}})=\overline{\oh_{K, p}}^{\times, \rho-1}$ by the discussion in Section \ref{S:PoincareTorsor}. The neutral element $1$ in these fibers thus makes sense, and, e.g., there is a unique $\xi_{i, j} \in \G_m^{\rho-1}(\overline{\oh_{K, p}})$ such that $\xi_{i, j}\overline{P_{i, j}}=1$, and we set $P'_{i, j}=\iota(\xi_{i, j})P_{i, j}$, where $\iota$ is the lift \eqref{eq:liftiota}. One defines points $R'_i, S'_j \in \PP^{\times, \rho-1}(\oh_{K, p})$ in a similar fashion. We similarly modify the points $V_{k, l}$ and $W_{k, l, i}$ of Notation \ref{not:VKL}. Alternatively, one can multiply the chosen basis of the torsion-free part of $\oh_K$-units $\mathbf{u}=\{u_1, \dots, u_{\delta}\}$ by suitable roots of unity of prime-to-$p$ order in $\oh_{K, p}$, so that the resulting units are congruent to $1$ modulo $p\oh_{K, p}$.

Using these points, we define 
\[
A'(\underline{n})={\sum}_{2, j}n_j \cdot_2 S'_j,\;\;\;\; B'(\underline{n})={\sum}_{1, i}n_i \cdot_1 R'_i,\;\;\;\; C'(\underline{n})={\sum}_{1, i}n_i \cdot_1 \left( {\sum}_{2, j}n_j \cdot_2 P'_{i, j} \right),
\]
and 
\[
U'_{k, l}(\underline{n}):=V'_{k, l}+_1 \sum_{1, i}n_i\cdot_1 W'_{k, l, i}.
\]

\begin{proposition}\label{mapE'}
With the above notations, given $(\underline{m}, \underline{n})\in \Z^{\delta(\rho-1)+r}$, we have
\[
E'(\underline{m},\underline{n})=\left({\sum}_{2, k, l}m_{k, l}\cdot_2U'_{k, l}(\underline{n})\right)+_2 \left( \big(C'(\underline{n})+_2B'(\underline{n})\big)+_1\big(A'(\underline{n})+_2\tilde{t}\big) \right).
\]
\end{proposition}

\begin{proof}
A formal computation similar to the proof of Proposition~\ref{prop:q-1-kills} shows that the expression defined by the right hand side lies in $\TT(\oh_{K, p})_{t}$. Since this expression is obtained by the same operations in terms of $+_1, +_2, \cdot_1,$ and $\cdot_2$ as $E(\underline{m}, \underline{n})$ (see \eqref{explicitE}), apart from the $\iota(\G_m^{\rho-1}(\overline{\oh_{K, p}}))$-action modification of the initial points, it follows from an analogue of (\ref{GmAndGrpLawsCommute}) that the right hand side expression differs from $E(\underline{m}, \underline{n})$ only by $\iota(\G_m^{\rho-1}(\overline{\oh_{K, p}}))$-action modification. In other words, the right hand side must equal $\xi'(\underline{m}, \underline{n})E(\underline{m}, \underline{n})$
for some $\xi'(\underline{m}, \underline{n}) \in \iota(\G_m^{\rho-1}(\overline{\oh_{K, p}}))$. Using the uniqueness of $\xi(\underline{m}, \underline{n})$ in Proposition \ref{adjustE}, this proves the indicated equality.
\end{proof}

\subsection{The $p$-adic interpolation}\label{s:interpolation}

The remaining part of this section aims to prove Theorem \ref{thm:pinterpolation}. This is done along the same lines as \cite[\S 3 \& \S 5.1]{EL19}, but in a more general context. We will use the following result, whose proof will be given shortly, to deduce Theorem \ref{thm:pinterpolation}.  

\begin{proposition}\label{cor:AllPrimesInterpolation}
\hfill
\begin{enumerate}
\item Let $X$ and $Y$ be smooth schemes over $\oh_K$ of relative dimensions $m$ and $n$ respectively. Let $f: X \lra Y$ be a morphism of $\oh_K$-schemes and let $x \in X(\overline{\oh_{K,p}})$ be a point. 
Any choice of local parameters (see Notation \ref{localpar}) followed by restriction of scalars (see Remark \ref{FunctionRestriction}) induces identifications $X(\oh_{K,p})_x \simeq \Z_p^{dm}$ and $Y(\oh_{K,p})_{f(x)} \simeq \Z_p^{dn}$
such that the composition of maps 
$$\Z_p^{dm}\simeq X(\oh_{K, p})_x \overset{f}{\lra} Y(\oh_{K, p})_{f(x)}\simeq \Z_p^{dn}$$ 
is given by convergent power series with $\Z_p$-coefficients. 

\item Let $G\lra Y$ be a smooth group scheme with identity section $e$, where $Y$ is smooth over $\oh_K$. Let $y \in Y(\overline{\oh_{K, p}})$ be a point. The map $$\mathbb{Z} \times G(\oh_{K, p})_{e(y)}\lra G(\oh_{K, p})_{e(y)}, \qquad (z, g) \mapsto z \cdot g$$ extends to a map $\mathbb{Z}_p \times G(\oh_{K, p})_{e(y)}\lra G(\oh_{K, p})_{e(y)}$, which describes the $\Z_p$-module action on fibers over $Y(\oh_{K, p})_y$, and which is given by convergent power series with $\Z_p$-coefficients after choosing local parameters for $G\lra Y$ at $e(y)$ and restricting scalars.
\end{enumerate}
\end{proposition}

\begin{remark}
We postpone the proof of this result to the end of this section. The proof of $(1)$ relies on the description of local parameters at a point $x$ using blow-ups. The proof of $(2)$ uses the formal logarithm and exponential maps to interpret the action $z\cdot g$ as $\exp(z\cdot \log(g))$. Since $\exp$ and $\log$ are given by convergent power series (see Proposition \ref{logexp} below), one can extend $z\cdot g$ to allow $\Z_p$-coefficients. The proofs are technical and quite general. For the sake of clarity and readability, we have chosen to defer them to after the proof of Theorem \ref{thm:pinterpolation}.
\end{remark}

\begin{proof}[Proof of Theorem \ref{thm:pinterpolation}]
By Proposition \ref{mapE'}, the map $E'$ is described by the operations $+_1$ and $+_2$, and their iterates $\cdot_1$ and $\cdot_2$. 
Proposition \ref{cor:AllPrimesInterpolation} (2) applied respectively to the group schemes $\PP^{\times} \lra \Jo$ and $\PP^{\times} \lra \J$, implies that the operations $(n, g)\mapsto n\cdot_1 g$ and $(n, g)\mapsto n\cdot_2 g$ for $n\in \Z$ extend to $n\in \Z_p$, and the resulting operations are given by convergent power series with $\Z_p$-coefficients. The expression in Proposition \ref{mapE'} therefore makes sense with $(\underline{m}, \underline{n})\in \Z_p^{\delta(\rho-1)}\times \Z_p^{r}$. Allowing for $\Z_p$-coefficients using the extended actions $\cdot_1$ and $\cdot_2$ thus gives rise via the formula in Proposition \ref{mapE'} to the desired map 
\[
\kappa:=\kappa_{\mathbf{x}, \mathbf{u}}: \Z_p^{\delta(\rho-1)+r} \lra \TT(\oh_{K,p})_t,
\]
\[
\kappa(\underline{m},\underline{n}):=\left({\sum}_{2, k, l}m_{k, l}\cdot_2U'_{k, l}(\underline{n})\right)+_2 \left( \big(C'(\underline{n})+_2B'(\underline{n})\big)+_1\big(A'(\underline{n})+_2\tilde{t}\big) \right),
\]
depending on the bases $\mathbf{x}$ of $\J(\oh_K)_0$ and $\mathbf{u}$ of $\G_m(\oh_K)_{\tf}$, and which by definition agrees with $E'$ when restricted to $\Z^{\delta(\rho-1)+r}\subset \Z_p^{\delta(\rho-1)+r}$.

By Proposition \ref{cor:AllPrimesInterpolation} (1), both operations 
\begin{align*}
& +_1: \PP^{\times, \rho-1}\times_{(\Jo)^{\rho-1}}\PP^{\times, \rho-1} \lra \PP^{\times, \rho-1},\;\; \\ 
& +_2: \PP^{\times, \rho-1}\times_{\J}\PP^{\times, \rho-1} \lra \PP^{\times, \rho-1}
 \end{align*} 
induce maps given by convergent power series over $\Z_p$ on the appropriate residue disks after choosing a regular system of local parameters inducing $\PP^{\times, \rho-1}(\oh_{K, p})_x \simeq \Z_p^{d(g+g(\rho-1)+\rho-1)}$ upon restricting scalars from $\oh_{K,p}$ to $\Z_p$. 
Since the composition of convergent power series with $\Z_p$-coefficients again produces convergent power series with $\Z_p$-coefficients, the map $\kappa$ is indeed given by a tuple of convergent $p$-adic power series.
\end{proof}

\subsubsection{Local parameters and blow-ups}

\begin{notation}\label{localpar}
We fix a prime $\mathfrak{p}\in \{\mathfrak{p}_1, \dots, \mathfrak{p}_s\}$ above $p$. Denote by $\pi$ a  uniformizer of $\oh_{K, \mathfrak{p}}$.
Let $X$ be a smooth scheme over $\oh_{K, \mathfrak{p}}$ of relative dimension $m$. Similarly to before, for a point $x \in X(\F_{\mathfrak{p}})$, we denote by $X(\oh_{K, \mathfrak{p}})_x$ the set of all $\oh_{K, \mathfrak{p}}$-points reducing to $x$ modulo $\mathfrak{p}$. 
By smoothness, the maximal ideal $\mathfrak{m}_x$ admits a regular system of parameters $(\pi, t_1, t_2, \dots, t_m)$. 
\end{notation}

The point $x$ factors through the natural map $\spec{\widehat{\oh_{X, x}}} \lra X$, and $X(\oh_{K, \mathfrak{p}})_x$ bijectively corresponds to $\spec{\widehat{\oh_{X, x}}}(\oh_{K,\mathfrak{p}})_x$. The isomorphism $  \oh_K[[t_1, \dots, t_m]]\simeq \widehat{\oh_{X, x}}$ then shows that there is a bijection of sets
\begin{align*}
t=(t_1, t_2, \dots, t_m):X(\oh_{K, \mathfrak{p}})_x &\stackrel{\sim}{\lra}  (\mathfrak{m}_{K, \mathfrak{p}})^{m} \\ 
\tilde{x}\;\; &\longmapsto  (t_1(\tilde{x}), \dots, t_m(\tilde{x})).
\end{align*}  
After dividing by $\pi$, one gets
\begin{equation} \label{eq:tilde_t} \tilde{t}=\left(\frac{t_1}{\pi}, \frac{t_2}{\pi}, \dots, \frac{t_m}{\pi}\right):X(\oh_{K, \mathfrak{p}})_x \stackrel{\sim}{\lra}  (\oh_{K, \mathfrak{p}})^{m}\;.
\end{equation} 

Let $f: X \lra Y$ be a morphism of schemes, which are smooth over $\oh_{K, \mathfrak{p}}$ of relative dimensions $m$ and $n$ respectively. Denote the analogous choice of a regular system of parameters of $Y$ at $f(x)$ by $s_1, s_2, \dots, s_{n}$, and the corresponding bijection by $$\tilde{s}: Y(\oh_{K, \mathfrak{p}})_{f(x)} \lra (\oh_{K, \mathfrak{p}})^{n}.$$ The immediate goal is the following.

\begin{proposition}\label{prop:blowup}
In the above setting, the composition $$f':(\oh_{K, \mathfrak{p}})^m\stackrel{\tilde{t}^{-1}}{\lra} X(\oh_{K, \mathfrak{p}})_x \stackrel{f}{\lra} Y(\oh_{K, \mathfrak{p}})_{f(x)} \stackrel{\tilde{s}}{\lra}(\oh_{K, \mathfrak{p}})^{n}$$
is given by a $n$-tuple of convergent power series with coefficients in $\oh_{K, \mathfrak{p}}$.
\end{proposition}

Here by convergent power series we mean elements of $\oh_{K, \mathfrak{p}}\langle X_1, X_2, \dots, X_m\rangle$, the $p$-adic, or equivalently $\pi$-adic, completion of $\oh_{K, \mathfrak{p}}[X_1, X_2, \dots, X_m]$. The proof of Proposition \ref{prop:blowup} follows closely \cite[\S 3]{EL19} by investigating the geometry of the situation.

\begin{proof}
By shrinking $X$ to a sufficiently small affine open neighbourhood of $x$, we may assume that $t_1, t_2, \dots, t_m$ are regular global functions, which define an \'{e}tale map $$t=(t_1, t_2, \dots, t_m): X \lra \mathbb{A}^m_{\oh_{K, \mathfrak{p}}}=\spec\oh_{K, \mathfrak{p}}[X_1, \dots, X_m],$$ mapping $x$ to the origin over $\F_{\mathfrak{p}}$, i.e., the point corresponding to $(\pi, X_1, \dots, X_d)$. By possibly shrinking $X$ further, we may assume that $x$ is in fact the only preimage of the origin. 

Note that a point $\tilde{x}: \spec{\oh_{K, \mathfrak{p}}}\lra X$ reduces to $x$ if and only if the pullback of $x$ along $\tilde{x}$ is the effective Cartier divisor cut out by $\pi$. Consequently, the universal property of the blowup $\mathrm{Bl}_x X$ of $X$ at $x$ implies that every $\tilde{x} \in X(\oh_{K,\mathfrak{p}})_x$ factors uniquely through $\mathrm{Bl}_x X$, and more precisely through the open subscheme $\mathrm{Bl}_x^{(\pi)} X$ of $\mathrm{Bl}_x X$, where $\pi$ is the generator of the exceptional divisor. 
Thus, we have a natural identification between $X(\oh_{K,\mathfrak{p}})_x$ and $\mathrm{Bl}_x^{(\pi)} X(\oh_{K, \mathfrak{p}})$. 

Up to this identification, the map $\tilde{t}$ can be described as follows. We consider the analogous construction for the $\F_{\mathfrak{p}}$-origin $o:\spec\F_{\mathfrak{p}}\lra \mathbb{A}^m_{\oh_{K, \mathfrak{p}}}$ from which we get $\mathrm{Bl}_o \mathbb{A}^m_{\oh_{K, \mathfrak{p}}}$ and 
$$\mathrm{Bl}_o^{(\pi)} \mathbb{A}^m_{\oh_{K, \mathfrak{p}}}=\spec\oh_{K, \mathfrak{p}}[\tilde{X}_1, \dots, \tilde{X}_m],$$ where in the expression above $\tilde{X}_i=X_i/\pi$. Since blowing up commutes with flat base change, we obtain a cartesian diagram of schemes
\begin{equation}\label{blowup}
\begin{tikzcd}
 \mathrm{Bl}_x^{(\pi)} X \arrow[hook]{r} \ar[d, "\tilde{t}"]\ar[dr, phantom, "\square"] & \mathrm{Bl}_x X  \ar[r] \arrow{d} \ar[dr, phantom, "\square"]
  & X \ar[d, "t"] \\
\mathrm{Bl}_o^{(\pi)} \mathbb{A}^m_{\oh_{K, \mathfrak{p}}}   \arrow[hook]{r}
  & \mathrm{Bl}_o \mathbb{A}^m_{\oh_{K, \mathfrak{p}}} \ar[r]
  & \mathbb{A}^m_{\oh_{K, \mathfrak{p}}}\;.
\end{tikzcd}  
\end{equation}
The map $\tilde{t}$ of (\ref{eq:tilde_t}) is simply the morphism $\tilde{t}$ in the above diagram evaluated at $\oh_{K, \mathfrak{p}}$-points. The notations are therefore compatible.


The map $\tilde{t}_{\F_{\mathfrak{p}}}$, obtained from base-changing the diagram (\ref{blowup}) to $\F_{\mathfrak{p}}$, can be non-canonically interpreted as the tangent map at $x$ between the respective tangent spaces. In particular, it is an isomorphism. Since $\tilde{t}$ is \'{e}tale, $\tilde{t}$ is an isomorphism when base-changed to ${\oh_{K, \mathfrak{p}}/(\pi^j)}$ for every $j$. Denoting the rings of global functions of the affine schemes in question by $\oh(\mathrm{Bl}_x^{(\pi)} X)$ and $\oh(\mathrm{Bl}_o^{(\pi)} \mathbb{A}^m_{\oh_{K,\mathfrak{p}}})$ respectively, we infer that their $\pi$-adic completions are the same, i.e.,
\begin{equation} \label{p-comp-iso}
\widehat{\oh(\mathrm{Bl}_x^{(\pi)} X)} \simeq \widehat{\oh(\mathrm{Bl}_o^{(\pi)} \mathbb{A}^m_{\oh_{K,\mathfrak{p}}})}=\widehat{\oh_{K, \mathfrak{p}}[\tilde{X}_1, \dots, \tilde{X}_m]}=\oh_{K, \mathfrak{p}}\langle\tilde{X}_1, \dots, \tilde{X}_m \rangle,
\end{equation}
both being equal to the algebra of integral formal power series converging on the unit disk.  

We perform the same analysis for $Y$, $f(x)$, and its fixed system of parameters $s_i$. Using again the universal property of the blowup of $Y$ at $f(x)$, we obtain that $f$ also induces a morphism 
$$\tilde{f}: \mathrm{Bl}_x^{(\pi)} X \lra \mathrm{Bl}_{f(x)}^{(\pi)} Y,$$ 
which on the level of $\oh_{K,\mathfrak{p}}$-points may be identified with $f: X(\oh_{K, \mathfrak{p}})_x\lra Y(\oh_{K, \mathfrak{p}})_{f(x)}$. Taking the $p$-adic completion of the associated ring map $\oh(\mathrm{Bl}_{f(x)}^{(\pi)} Y) \lra \oh(\mathrm{Bl}_{x}^{(\pi)} X)$ and conjugating by the isomorphisms (\ref{p-comp-iso}) for $X$ and $Y$ then yields a map 
$$\oh_{K, \mathfrak{p}}\langle\tilde{Y}_1, \dots, \tilde{Y}_{n} \rangle \lra \oh_{K, \mathfrak{p}}\langle\tilde{X}_1, \dots, \tilde{X}_m \rangle.$$
This is described by specifying $n$-tuple of elements of $\oh_{K, \mathfrak{p}}\langle\tilde{X}_1, \dots, \tilde{X}_m \rangle$ as images of the variables $\tilde{Y}_i$.  Since the map $f'$ is obtained from the above map of rings by applying the functor $\mathrm{Hom}_{\mathrm{Alg}_{\oh_{K, \mathfrak{p}}}}(-, \oh_{K, \mathfrak{p}})$, it follows that $f'$ is described by these power series. This proves the claim. 
\end{proof}

\begin{remark}\label{rem:ParametrizeCoefficients}
It will be useful later to note that $\oh_{X, x}$ naturally embeds into $\widehat{\oh(\mathrm{Bl}_{x}^{(\pi)} X)}$. As in the previous proof, let us replace $X$ by a sufficiently small affine neighbourhood of $x$. The maximal ideal of $\oh_X(X)$ corresponding to $x$ becomes $(\pi)$ in $\oh(\mathrm{Bl}_{x}^{(\pi)} X)$, and is therefore mapped to the radical in $\widehat{\oh(\mathrm{Bl}_{x}^{(\pi)} X)}$. 
There is thus an induced map 
$\oh_{X, x} \lra \widehat{\oh(\mathrm{Bl}_{x}^{(\pi)} X)}.$ Concerning the injectivity, after taking completions at the maximal ideal, the map becomes the map 
$$\oh[[X_1, \dots, X_m]]\hookrightarrow \oh[[\tilde{X}_1, \dots, \tilde{X}_m]]$$ given by $X_i \mapsto p \tilde{X}_i$, which is injective.
\end{remark}



\begin{remark}\label{FunctionRestriction}
It will be beneficial to replace the power series expressions with $\oh_{K, \mathfrak{p}}$-coefficients by convergent power series with $\Z_p$-coefficients. To that end, we let 
$$k=ef=\mathrm{rank}_{\Z_p}\oh_{K, \mathfrak{p}},$$ following our earlier convention, and fix a free basis $e_1, e_2, \dots, e_k$ of $\oh_{K,\mathfrak{p}}$ as a $\Z_p$-module. Expressing everything with respect to this basis, the description of maps $\oh_{K,\mathfrak{p}}^{m}\lra \oh_{K,\mathfrak{p}}^{n}$ in terms of power series gives rise to a power series description of maps $\Z_p^{km}\lra \Z_p^{kn}$. More precisely, upon the introduction of formal variables $X_{i,j}$ by the rule
\begin{equation}\label{VariableRestriction}
X_i=X_{i,1}e_1+X_{i,2}e_2+\dots+X_{i,k}e_k,
\end{equation} 
any convergent power series $f \in \oh_{K,\mathfrak{p}}\langle X_1, X_2, \dots X_m \rangle$ can be written as 
$$f=f_1e_1+f_2e_2+\dots +f_ke_k,$$
for a unique $k$-tuple of power series $f_1, f_2, \dots, f_k \in \Z_p\langle X_{i,j}\;|\; 1\leq i \leq m, 1\leq j \leq k\rangle$.
\end{remark}

\begin{remark}\label{rem:InjectiveTangent}
Keeping the notation from the proof of Proposition~\ref{prop:blowup}, the map 
$$\tilde{f}_{\F_{\mathfrak{p}}}: (\mathrm{Bl}_x^{(\pi)} X)_{\F_{\mathfrak{p}}} \lra (\mathrm{Bl}_{f(x)}^{(\pi)} Y)_{\F_{\mathfrak{p}}}$$ can be identified with the tangent map of $f_{\F_{\mathfrak{p}}}:X_{\F_{\mathfrak{p}}} \lra Y_{\F_{\mathfrak{p}}}$ at $x$. Assume that this map is injective. By a lift of a suitable $\F_{\mathfrak{p}}$-affine change of coordinates on $(\mathrm{Bl}_{f(x)}^{(\pi)} Y)_{\F_{\mathfrak{p}}}$, one can ensure that the map 
$(f')^{\#}:\oh_{K, \mathfrak{p}}\langle\tilde{Y}_1, \dots, \tilde{Y}_n\rangle\lra \oh_{K, \mathfrak{p}}\langle\tilde{X}_1, \dots, \tilde{X}_m\rangle$ is given by $\tilde{Y}_i \mapsto \tilde{X}_i$ for $i\leq m$ and by $\tilde{Y}_i \mapsto 0$ for $i>m$. In other words, the parameters $s_i$ and $t_i$ may be chosen so that $f^{\#}(s_i)=t_i$ for $i\leq m$ and $s_{m+1}, \dots s_n$ generate the kernel of the map $f^{\#}: \oh_{Y, f(x)}\lra \oh_{X, x}$. In that case, $X(\oh_{K, \mathfrak{p}})_{x}$ is embedded in $Y(\oh_{K, \mathfrak{p}})_{f(x)}$ and, in the chosen coordinates, equal to the vanishing locus of $\tilde{Y}_{m+1}, \dots, \tilde{Y}_{n}$. As in Remark~\ref{FunctionRestriction}, we can identify the embedding  with the affine embedding $\Z_p^{km}\lra \Z_p^{kn}$, whose image is cut out by the $k(n-m)$ variables $\tilde{Y}_{i, j},$ with $m < i \leq n$ and $1\leq j \leq k$. 
\end{remark}

\subsubsection{The exp-log argument} 

We now focus on the special case where $Y \lra \spec\oh_{K, \mathfrak{p}}$ is a smooth scheme of relative dimension $n$, and  $X=G$ is a smooth commutative group scheme over $Y$ of relative dimension $m$. Beware that $m$ from the previous discussion corresponds to $m+n$ in the situation at hand. Hopefully, this will not cause too much confusion. Let $e: Y \lra G$ denote the identity section. We consider a point $y \in Y(\F_{\mathfrak{p}})$ and the map $G(\oh_{K, \mathfrak{p}})_{e(y)}\lra Y(\oh_{K, \mathfrak{p}})_{y}$.

As done previously, we may replace $Y$ by $\spec{\oh_{Y, y}}$ and $G$ by $G_{\oh_{Y, y}}$. We fix a system of parameters $\pi, s_1, s_2, \dots, s_n$, which induce a bijection $\tilde{s}: Y(\oh_{K, \mathfrak{p}})_{y}\stackrel{\sim}{\lra} \oh_{K, \mathfrak{p}}^n$.

By \cite[05D9]{stacks}, there is an affine open neighborhood $\spec{B}=U \subseteq G_{\oh_{Y, y}}$ of $e(y)$ such that $e$ factors through $U$ and such that, denoting by $I$ the kernel of the associated map $e^{\#}: B \lra \oh_{Y, y}$, $I/I^2$ is a free $\oh_{Y, y}$-module of rank $m$. Upon fixing elements $t_1, t_2, \dots, t_m \in I$ that form a free basis of $I/I^2$, the sequence $\pi, s_1, s_2, \dots, s_n, t_1, t_2, \dots, t_m$ forms a system of parameters of $G_{\oh_{Y, y}}$ at $e(y)$. This establishes a bijection $(\tilde{s}, \tilde{t}): G(\oh_{K, \mathfrak{p}})_{e(y)} \stackrel{\sim}{\lra} \oh_{K, \mathfrak{p}}^{n+m}.$ 

We consider the formal $\oh_{Y, y}$-group $\widehat{G_{\oh_{Y, y}}}$, i.e., the completion of $G_{\oh_{Y, y}}$ with respect to the ideal of the identity section. In terms of the chosen coordinates, it is the formal spectrum of the $I$-adic completion of $B$, which in turn is the formal power series ring $\oh_{Y, y}[[t_1, t_2, \dots, t_m]]$. The group operation then induces a $m$-dimensional commutative formal group law $$F_G(\underline{U}, \underline{V})=(F_1, \dots, F_m)(U_1, \dots, U_m, V_1, \dots, V_m)$$ over $\oh_{Y, y},$ in the sense of \cite{honda}.  By \cite[Theorem 1]{honda}, over $\oh_{Y, y}\otimes \Q,$ there are mutually inverse isomorphisms of formal group laws 
\begin{center}
\begin{tikzcd}
 F_{G, \Q} \ar[r, shift left, "\mathrm{log}"] & \ar[l, shift left, "\mathrm{exp}"] (\widehat{\mathbb{G}}_a)_{\Q}^m, 
\end{tikzcd}  
\end{center}
where $(\widehat{\mathbb{G}}_a)^m$ denotes the $m$-dimensional addition law given by the polynomials $U_i+V_i$ treated as power series over $\oh_{Y, y}$. The subscript $\Q$ denotes the ``formal base change'' to $\Q$. Explicitly, fixing a basis of invariant differentials $\omega_1, \dots, \omega_m \in \bigoplus_{i=1}^m \oh_{Y, y}[[t_1, \dots, t_m]] \mathrm{d}t_i$ of $F_G$ in the sense of \cite[Proposition 1.1]{honda}, $\mathrm{log}$ is given by a $m$-tuple of formal power series $$L_1, L_2, \dots, L_m \in (\oh_{Y, y}\otimes\Q)[[t_1, \dots, t_m]]$$ characterized by the properties
\begin{equation}\label{logdef}
L_i(0, \dots, 0)=0,\;\; \mathrm{d} L_i=\omega_i,\;\;i=1, 2, \dots, m,
\end{equation}
and that, additionally, each $L_i$ equals $t_i$ in degree $1$.
The exponential is then given as a formal inverse to $\mathrm{log},$ i.e., by a $m$-tuple of power series $E_1, E_2, \dots, E_m \in (\oh_{Y, y}\otimes\Q)[[t_1, \dots, t_m]]$ characterized by the identities
\begin{equation}\label{expdef}
E_i(L_1, L_2, \dots, L_m)=t_i,\;\; i=1, 2, \dots, m.
\end{equation}
It follows that each $E_i$ equals $t_i$ in degrees $\leq 1$.

The fibers of the map $G(\oh_{K, \mathfrak{p}})_{e(y)}\lra Y(\oh_{K, \mathfrak{p}})_{y}$ naturally carry the structure of $\mathbb{Z}_p$-modules. In fact, given $\tilde{y}\in Y(\oh_{K, \mathfrak{p}})_{y},$ the fiber over $\tilde{y}$ is the kernel of the reduction map $G_{\tilde{y}}(\oh_{K, \mathfrak{p}})\lra G_{\tilde{y}}(\mathbb{F}_{\mathfrak{p}})$, where $G_{\tilde{y}}$ denotes the $\oh_{K, \mathfrak{p}}$-group scheme obtained from $G$ by base change along $\tilde{y}$. This kernel is the set of $\oh_{K, \mathfrak{p}}$-points of the associated formal group, $\widehat{G_{\tilde{y}}}(\oh_{K, \mathfrak{p}})=\varprojlim_{j}\widehat{G_{\tilde{y}}}(\oh_{K, \mathfrak{p}}/p^j\oh_{K, \mathfrak{p}})$, and the group law of $\widehat{G_{\tilde{y}}}$ may be viewed as the ``formal base change'' of the formal group law for $\widehat{G_{\oh_{Y, y}}}$ above. The fact that any formal group law is of the form $\underline{U}+\underline{V}+(\text{higher order terms})$ shows that $\widehat{G_{\tilde{y}}}(\oh_{K, \mathfrak{p}}/p^j\oh_{K, \mathfrak{p}})$ is an abelian group annihilated by $p^j$. This verifies the claim.

The goal is to $p$-adically interpolate the function $z \mapsto z\cdot g$ for $g \in G(\oh_{K, \mathfrak{p}})_{e(y)}$, or more precisely, to describe the action map $\Z_p\times G(\oh_{K, \mathfrak{p}})_{e(y)}\lra G(\oh_{K, \mathfrak{p}})_{e(y)}$ arising from the $\Z_p$-action on the fibers, in terms of convergent power series. This is done by interpreting the formal logarithm and exponential as convergent power series.

\begin{proposition}\label{logexp}
The formal logarithm and exponential induce the mutually inverse maps $\mathrm{log}$ and $\mathrm{exp}$
\begin{center}
\begin{tikzcd}
G(\oh_{K, \mathfrak{p}})_{e(y)} \ar[r, "{(\tilde{s}, \tilde{t})}", "\simeq"'] & (\oh_{K, \mathfrak{p}})^{n+m} \ar[r, shift left, "\mathrm{log}"] & \ar[l, shift left, "\mathrm{exp}"] (\oh_{K, \mathfrak{p}})^{n+m},
\end{tikzcd}  
\end{center} 
which are given by convergent power series, i.e., elements of $\oh_{K, \mathfrak{p}}\langle \tilde{Y}_1, \dots, \tilde{Y}_n,\tilde{X}_1, \dots, \tilde{X}_m \rangle$.
Given $z \in \Z_p$ and $g \in G(\oh_{K, \mathfrak{p}})_{e(y)}$, viewed as an element of $(\oh_{K, \mathfrak{p}})^{n+m}$ via $(\tilde{s}, \tilde{t})$, we have the equality $z \cdot g=\mathrm{exp}(z \cdot\mathrm{log}(g))$. Consequently, the action map $\mathbb{Z}_p \times G(\oh_{K, \mathfrak{p}})_{e(y)} \lra G(\oh_{K, \mathfrak{p}})_{e(y)}$ is described by convergent power series with coefficients in $\Z_p$.
\end{proposition}

\begin{proof}
Write $L_i=\sum_{J \neq 0}a_{i, J}\underline{t}^J$ and $E_i=\sum_{J \neq 0}b_{i, J}\underline{t}^J$ for the formal power series, which are respectively the components of the formal logarithm and formal exponential. It can be deduced from the identity (\ref{logdef}) that
\begin{equation}\label{logcoeff}
|J|a_{i, J} \in \oh_{Y, y}, \qquad \text{ for all }J,
\end{equation}
and a formal computation of the exponential based on the identities (\ref{expdef}) as in \cite[A.4.6]{hazewinkel}, together with (\ref{logcoeff}), shows that 
\begin{equation}\label{expcoeff}
(|J|!)b_{i, J} \in \oh_{Y, y}, \qquad \text{ for all }J.
\end{equation}
The induced map $\mathrm{log}: \oh_{K, \mathfrak{p}}^{n+m}\lra \oh_{K, \mathfrak{p}}^{n+m}$ is then given by the identity on the first $n$ components, which correspond to the base $Y(\oh_{K, \mathfrak{p}})_y$, and by the power series 
\begin{equation}\label{LogOnDisks}
\tilde{L}_i(\underline{\tilde{X}})=\pi^{-1} \sum_{J \neq 0}a_{i, J}(\pi \underline{\tilde{X}})^J=\sum_{J \neq 0}\frac{\pi^{|J|-1}}{|J|}(|J|a_{i, J})(\underline{\tilde{X}})^J, \qquad i=1, \dots, m
\end{equation}
on the remaining components. Here $|J|a_{i, J}$ is considered as an element of $\oh_{K, \mathfrak{p}}\langle \tilde{Y}_1, \dots, \tilde{Y}_m\rangle$ in the sense of Remark~\ref{rem:ParametrizeCoefficients}.

Its formal inverse is given by the analogous modification of the formal exponential. Namely, $\mathrm{exp}: \oh_{K, \mathfrak{p}}^{n+d}\lra \oh_{K, \mathfrak{p}}^{n+d}$ is given by the identity on the first $n$ components and on the remaining $m$ components by the formal power series
\begin{equation}\label{ExpOnDisks}
\tilde{E}_i(\underline{\tilde{X}})=\pi^{-1} \sum_{J \neq 0}b_{i, J}(\pi \underline{\tilde{X}})^J=\sum_{J \neq 0}\frac{\pi^{|J|-1}}{|J|!}((|J|!)b_{i, J})(\underline{\tilde{X}})^J, \qquad i=1, \dots, m,
\end{equation}
where $(|J|!)b_{i, J}$ is again considered as an element of $\oh_{K, \mathfrak{p}}\langle \tilde{Y}_1. \dots, \tilde{Y}_m \rangle$.

To conclude that the power series (\ref{LogOnDisks}) and (\ref{ExpOnDisks}) define elements of the ring $\oh_{K, \mathfrak{p}}\langle \underline{\tilde{Y}}, \underline{\tilde{X}}\rangle$, it is enough to observe that the coefficients $\pi^{|J|-1}/(|J|!)$, and thereby also $\pi^{|J|-1}/|J|$, are integral and converge to zero $p$-adically as $|J|\rightarrow \infty$. These properties hold thanks to the imposed condition $e < p-1$ in Assumption \ref{passumption} $(b)$ on the ramification index, since the $p$-adic valuations satisfy
$$v_p\left(\frac{\pi^{k-1}}{k!}\right)\geq \frac{k-1}{e}-\frac{k-1}{p-1}=\frac{(k-1)(p-1-e)}{e(p-1)},$$
which is then non-negative for all $k\geq 1$ and tends to $\infty$ as $k \rightarrow \infty$. 

Finally, we may interpret $\mathrm{log}$ and $\mathrm{exp}$ as given by $ef(n+m)$ power series with coefficients in $\Z_p$ following Remark~\ref{FunctionRestriction}. The action map $\Z_p \times G(\oh_{K, \mathfrak{p}})_{e(y)} \lra G(\oh_{K, \mathfrak{p}})_{e(y)}$ then becomes a $p$-adically continuous map $\Z_p \times \Z_p^{ef(n+m)}\lra \Z_p^{ef(n+m)}$, which extends the map $$(z, g)\mapsto z\cdot g=\mathrm{exp}(z \cdot \mathrm{log}(g))$$ from $\Z \times \Z_{p}^{ef(n+m)}$ to $\Z_p \times \Z_{p}^{ef(n+m)}$. The same is true concerning the map on $\Z_p \times \Z_{p}^{ef(n+m)}$ given by $(z, g)\mapsto \mathrm{exp}(z \cdot \mathrm{log}(g))$, and it follows that these two maps agree. In particular, the $\Z_p$-action map is described by convergent power series with $\Z_p$-coefficients as claimed.
\end{proof}

\subsubsection{Proof of Proposition \ref{cor:AllPrimesInterpolation}}

\begin{proof}
As in (\ref{OKp-OKP-points}), a point $x \in X(\overline{\oh_{K,p}})$ is given by a $s$-tuple $x_1 \in X(\F_{\mathfrak{p}_1}), \dots, x_s \in X(\F_{\mathfrak{p}_s}),$ and we have $X(\oh_{K,p})_x=\prod_{i=1}^sX(\oh_{K,\mathfrak{p}_i})_{x_i}$. Similarly, for any map $f: X \lra Y$ of $\oh_{K}$-schemes, the induced map $f: X(\oh_{K, p})_x \lra Y(\oh_{K, p})_{f(x)}$ decomposes into the product over $i$ of the maps $f: X(\oh_{K, \mathfrak{p}})_{x_i} \lra Y(\oh_{K, \mathfrak{p}})_{f(x_i)}$. Part (1) thus follows from Proposition~\ref{prop:blowup} and Remark~\ref{FunctionRestriction}.

Likewise, we have $G(\oh_{K,p})_{e(y)}=\prod_{i=1}^sG(\oh_{K,\mathfrak{p}_i})_{e(y_i)},$ and thus $G(\oh_{K,p})_{e(y)}$ has a $\Z_p$-module structure on fibers over $Y(\oh_{K,p})_y=\prod_{i=1}^sY(\oh_{K,\mathfrak{p}_i})_{y_i}.$ By Proposition~\ref{logexp}, each of the action maps $\mathbb{Z}_p\times G(\oh_{K,\mathfrak{p}_i})_{e(y_i)}\lra G(\oh_{K,\mathfrak{p}_i})_{e(y_i)}$ is given by convergent power series with $\Z_p$-coefficients. The action map for $G(\oh_{K,p})_{e(y)}$ is then obtained by taking the product of the above action maps and precomposing with $\Z_p \times G(\oh_{K,p})_{e(y)}\lra \prod_i (\Z_p \times  G(\oh_{K,\mathfrak{p}_i})_{e(y_i)})$, where $\Z_p$ is embedded into the $s$ copies of $\Z_p$ diagonally. It follows that the map has a description in terms of convergent power series over $\Z_p$ as well, which proves (2).
\end{proof}


\section{Bounding the number of rational points}\label{s:end}



We are now in a position to prove Theorem \ref{thm:main_rough_form} of Section \ref{s:method}, which gives a conditional upper bound on the size of the intersection $\U(\oh_{K, p})_u\cap \Y_t$.

The tangent map of the lifted Abel--Jacobi map $\tilde{j}_b^U : \U(\oh_{K,\fp})_{u_\fp}\hookrightarrow \TT(\oh_{K, \fp})_{t_\fp}$ of Proposition \ref{prop:Ulift} is injective at $\fp$ by smoothness (Assumption \ref{passumption} (a)). It follows, by Remark \ref{rem:InjectiveTangent}, that $\U(\oh_{K, \fp})_{u_\fp}$ is a complete intersection in $\TT(\oh_{K,\mathfrak{p}})_{t_{\mathfrak{p}}}$, i.e., it is cut out by $g+\rho-2$ elements 
\[
f_1^{\fp}, \ldots, f_{g+\rho-2}^{\fp} \:\: \in \:\: \widehat{\oh(\Bl_{t_\fp}^{(\pi_{\fp})}(\TT))},
\]
which generate the kernel of the surjection 
$$(\tilde{j}^U_b)^{\#}_{\fp} : \widehat{\oh(\Bl_{t_\fp}^{(\pi_{\fp})}(\TT))}\lra \widehat{\oh(\Bl_{u_\fp}^{(\pi_{\fp})}(\U))}.$$
After restricting scalars from $\oh_{K,\fp}$ to $\Z_p$ (see Remark \ref{FunctionRestriction}), each $f_i^{\fp}$ corresponds uniquely to a $k_{\fp}$-tuple $(f_{i,1}^\fp, \ldots, f_{i,k_\fp}^\fp)$ of power series in $(g+\rho-1)k_\fp$ variables with coefficients in $\Z_p$. 

In conclusion, the $p$-adic analytic submanifold $\U(\oh_{K, p})_u$ of $\TT(\oh_{K, p})_t$ is cut out by $(g+\rho-2)\sum_{\fp\vert p} k_{\fp}=(g+\rho-2)d$ convergent power series with coefficients in $\Z_p$. 
The computation of the desired intersection is accomplished by pulling all equations back via the map of $p$-adic analytic manifolds $\kappa=\kappa_{\mathbf{x}, \mathbf{u}}: \Z_p^{\delta(\rho-1)+r}\lra \TT(\oh_{K,p})_t$ of Theorem \ref{thm:pinterpolation} whose image is $\Y_t$ by Corollary \ref{coro:imagekappa}.

\begin{definition}\label{AI}
The elements $\kappa^* f_{i, j}^{\fp}$, with $1\leq i\leq g+\rho-2$, $1\leq j\leq k_{\fp}$, and $\fp\vert p$, all lie in $R=\Z_p\langle z_1, \ldots, z_{\delta(\rho-1)+r} \rangle$. Let $I_{\U, u}:=I_{\U, u,\kappa}$ denote the ideal of $R$ generated by these elements, and denote by $A_{\U, u}:=R /I_{\U, u}$ the resulting quotient ring.
\end{definition}

The intersection is algebraically expressed as the tensor product of rings, i.e., by taking the quotient by $I_{\U, u}$. It follows that there is a bijection 
\begin{equation}\label{HomA}
    \Hom(A_{\U, u}, \Z_p) \longleftrightarrow \kappa^{-1}(\U(\oh_{K, p})_u\cap \Y_t).
\end{equation}

Let $\overline{\kappa^* f_{i, j}^{\fp}}\in \F_p[z_1, \ldots, z_{\delta(\rho-1)+r}]$ denote the reductions of the convergent power series modulo $p$. These elements generate the ideal $\bar{I}_{\U, u}=I_{\U, u} \F_p[z_1, \ldots, z_{\delta(\rho-1)+r}]$. Consider the quotient $\F_p$-algebra
\[
\overline{A}_{\U, u}:=A_{\U, u}\otimes \F_p=\F_p[z_1, \ldots, z_{\delta(\rho-1)+r}]/\bar{I}_{\U, u}.
\]


We are now ready to prove Theorem \ref{thm:main_rough_form}, which we restate.\\

\noindent \textbf{Theorem} \ref{thm:main_rough_form}. 
{\it If $\overline{A}_{\U, u}$ is finite, then $\vert \U(\oh_K)_u\vert\leq \dim_{\F_p} \overline{A}_{\U, u}$.}

\begin{proof}
For the sake of notation, we drop the subscripts $(\U, u)$ in this proof.
The ring $A$ is $p$-adically complete by the same proof as of \cite[Theorem 4.12]{EL19}. Moreover, since $\overline{A}$ is finite, $A$ is finitely generated as a $\Z_p$-module. 
It follows that 
\[
\Hom(A, \Z_p)=\coprod_{\mathfrak{m}} \Hom(A_\mathfrak{m}, \Z_p)=\coprod_{A_\mathfrak{m}/\mathfrak{m}=\F_p} \Hom(A_\mathfrak{m}, \Z_p),
\]
where the union is over the maximal ideals of $A$.
This gives the bound 
\[
\vert \Hom(A, \Z_p)\vert \leq \sum_{A_\mathfrak{m}/\mathfrak{m}=\F_p} \rank_{\Z_p} A_\mathfrak{m}=\sum_{A_\mathfrak{m}/\mathfrak{m}=\F_p} \dim_{\F_p} \overline{A}_\mathfrak{m}\leq \dim_{\F_p} \overline{A}.
\]
This establishes, by \eqref{HomA}, that the number of points in $\kappa^{-1} (\U(\oh_{K, p})_u \cap \Y_t )$ is bounded by $\dim_{\F_p} \overline{A}$, and we thus have 
\[
\vert \U(\oh_K)_u \vert\leq \vert \kappa^{-1} (\U(\oh_{K, p})_u \cap \Y_t )\cap \overline{T(\oh_K)_t}^p)\vert \leq \dim_{\F_p} \overline{A}.
\]
\end{proof}

\begin{remark}
The geometric quadratic Chabauty condition is implicit in the assumption of Theorem \ref{thm:main_rough_form}. Indeed, in order for the ring 
$
\overline{A}_{\U, u}=\F_p[z_1, \ldots, z_{\delta(\rho-1)+r}]/\langle \overline{\kappa^* f_{i, j}^{\fp}} \rangle_{i, j, \fp}
$
to have a chance to be finite, the number of relations we quotient by must be at least the number of variables. We therefore need $\delta(\rho-1)+r\leq (g+\rho-2)d$, which is equivalent to \eqref{eq:chabbound}.
\end{remark}

\begin{corollary}\label{cor:GQCNF}
Suppose that $\overline{A}_{\U, u}$ is finite for all $\U$ as in Definition \ref{cons:U} and all $u\in \U(\overline{\oh_{K,p}})$. Then the set of rational points $C_K(K)$ is finite and satisfies
\[
\vert C_K(K) \vert \leq \sum_\U \sum_{u\in \U(\overline{\oh_{K,p}})} \dim_{\F_p} \overline{A}_{\U, u}.
\]
\end{corollary}

\begin{proof}
There are finitely many $\U\subset \CC^{\sm}$ satisfying the conditions of Definition \ref{cons:U} and the union of the $\U(\oh_{K})$ covers $\CC^{\sm}(\oh_K)$, which is equal to $C_K(K)$ by properness and regularity of the model $\CC$. Moreover, each $\U(\oh_{K})$ is the disjoint union of its residue disks $\U(\oh_{K})_u$. The result follows.
\end{proof}

\section{Discussion and questions}\label{s:question}

\subsection{Finiteness of intersections}\label{ss:inter}

A more precise form of Question \ref{q1} from the introduction is the following:
\begin{question}\label{q3}
Given a subscheme $\U$ as in Definition \ref{cons:U} and $u\in \U(\overline{\oh_{K,p}})$ mapping to the point $\tilde{j}_b^U(u)=t\in \TT(\overline{\oh_{K,p}})$, what conditions would guarantee the finiteness of the intersection $\Y_t \cap \U(\oh_{K, p})_u$? 
\end{question} 

In \cite[\S 9]{EL19}, Edixhoven and Lido give a new proof of Faltings' theorem using their method in the case of higher genus curves defined over $\Q$ satisfying $r<g+\rho-1$. Their argument is quite elegant: it uses complex analytic methods to prove a Zariski density statement, which can then be bridged with their $p$-adic geometric situation using formal geometry. This proves the finiteness of the intersection $\Y_t \cap \U(\Z_p)_u$ and, in particular, the finiteness of $C_{\Q}(\Q)$. 

The setting over arbitrary number fields is more complicated. Reminiscent of the failures of Siksek's method described in Section \ref{s:intro:RoSC}, there are examples of curves satisfying \eqref{eq:chabbound} for which the  intersection $\Y_t \cap \U(\oh_{K, p})_u$ is not finite. Examples include curves base changed from $\Q$, which do not satisfy the quadratic Chabauty condition over $\Q$. Based on the results of Dogra \cite{Dogra19} (discussed in Section \ref{s:intro:RoSQC}), it seems reasonable to expect (although we have no proof) that the intersection is finite under the conditions \eqref{eq:chabbound} and 
\begin{equation}\label{dogra}
\Hom(J_{\bar{\Q}, \sigma_1}, J_{\bar{\Q}, \sigma_2})=0  \text{ for any two distinct embeddings } \sigma_1, \sigma_2 : K\hookrightarrow \bar{\Q}.
\end{equation}
It should be added that even if the condition \eqref{eq:chabbound} fails, we do expect a similar outcome as long as $\dim \Y_t < (g+\rho-2)d$ holds. 

A natural further question is to find a necessary and sufficient condition for the finiteness of the intersection. At this stage, we have no insights to offer in this direction.

\subsection{A finite-to-one condition}


Assume the finiteness of the intersections $\Y_t \cap \U(\oh_{K, p})_u$ (for the open subschemes $\U$ of Definition \ref{cons:U} and each point $u \in \U (\overline{\oh_{K, p}})$) discussed in Section \ref{ss:inter}. In order to extract an explicit bound for $\vert C_{K}(K)\vert$, the method of this paper relies on the existence of a prime $p$ with the property that the $\F_p$-algebras $\overline A_{\U, u}$ of Definition \ref{AI} are finite dimensional. Over $\Q$, assuming $r \leq g+\rho-2,$ Edixhoven and Lido hope (but also expect) \cite[Remark 4.13]{EL19} that it is always possible in practice to choose $p$ with this property. The purpose of this section is to discuss what conditions would guarantee the existence of such a prime. 

Fix an open subscheme $\U$ and a point $u \in \U (\overline{\oh_{K, p}})$. Given the definition of the algebra $A_{\U,u}$, it is natural to tackle the question of $\overline A_{\U, u}$ being finite or not by investigating the closely related question of when the map $\kappa=\kappa_{\U,u} : \Z_p^{r+\delta(\rho-1)}\twoheadrightarrow \Y_t$ of Theorem \ref{thm:pinterpolation} is expected to have finite fibers.  Throughout the discussion, we assume the condition \eqref{eq:chabbound}. 

Let $d_J:=\dim \overline{\J(\oh_K)_x}^p$ and $d_T:=\dim \overline{\TT(\oh_K)_t}^p$ be the dimensions as $p$-adic analytic manifolds. Similarly, let $d_{\oh_K^{\times}}$ denote the $\Z_p$-module rank of the $p$--adic closure of the image of $\G_m(\oh_K)_1$ (where the subscript $1$ indicates principal units, i.e., taking the residue disk of $1 \in \oh_K\otimes \F_p$) embedded diagonally in $\G_m(\oh_{K, p})_1=\prod_{\mathfrak{p}\mid p} \G_m(\oh_{K,\mathfrak{p}})_1$. The following inequalities of dimensions hold: 
\begin{itemize}
    \item $d_{\oh_K^{\times}}\leq \delta,$ \hfill (an equality assuming Leopoldt's conjecture \cite{leopoldt} for $K$ and $p$)
    \item $d_J\leq \min(r, gd)$, 
    \item $d_T\leq \min(r+\delta(\rho-1), (g+\rho-1)d)$, 
    \item $d_J+d_{\oh_K^{\times}}(\rho-1)\leq d_T$.
\end{itemize}

In order for $\kappa$ to be finite-to-one, the equality of dimensions $d_T=r+\delta(\rho-1)$ is needed. Hence, under our current assumptions, we expect the geometric quadratic Chabauty method over number fields to produce an effective bound on the size of the set of rational points when $\rho\geq 2$ and 
\begin{equation}\label{condK}
 r+\delta(\rho-1)= d_T \le (g+\rho-2)d.
\end{equation}
This condition should perhaps aptly be called the {\it effective geometric quadratic Chabauty condition} in contrast to the weaker condition \eqref{eq:chabbound}.

It is then natural to ask the following:
\begin{question}\label{qGuidoNrFields}
Is it possible that $d_J+d_{\oh_K^{\times}}(\rho-1)<d_T$?
\end{question}

\begin{remark}
We currently have no definite answer to offer. However,
using the fact that the torsor $\mathbf{T}$ is set up in a manner such that $t^*\mathbf{T}$ is trivial over $\spec{\oh_K}$ for every $t \in \mathbf{J}(\oh_K)$, it can be shown that there exists a Zariski open cover $\mathbf{J}=\bigcup_i\mathbf{V}_i$ which is trivializing for $\mathbf{T}$ and such that $\mathbf{J}(\oh_K)=\bigcup_i\mathbf{V}_i(\oh_K)$ (i.e., every $\oh_K$--point factors through the open subscheme $\mathbf{V}_i$ for some $i$). If such a cover can additionally be taken to be finite, then the equality $d_T=d_J+d_{\oh_{K}^\times}(\rho-1)$ follows. Nonetheless, such finiteness appears (at least to us) to be a non-trivial condition on the torsor. \end{remark}

If the answer to the above question is negative, implying that $d_T=d_J+d_{\oh_K^{\times}}(\rho-1),$ then the condition $d_T=r+\delta(\rho-1)$ in \eqref{condK} would force $d_J=r$ and $d_{\oh_K^{\times}}=\delta$. In particular, we must have 
\begin{equation}\label{rANDgd}
r\le gd.
\end{equation}
In this case, the geometric quadratic Chabauty method is expected to yield an effective bound under the condition
\begin{equation}\label{chK}
    d_J+d_{\oh_K^{\times}}(\rho-1)=r+\delta(\rho-1)\le (g+\rho-2)d.
\end{equation}

\begin{remark}
Of course, within \eqref{rANDgd}, the most interesting case is the range
\begin{equation}\label{eq:resr}
    (g-1)d<r \leq gd,
\end{equation}
since for $r \leq (g-1)d $ Siksek's RoS linear Chabauty can be applied (with finiteness of the Chabauty set guaranteed by the work of Dogra discussed in Section \ref{s:intro:RoSQC}, under the additional assumption \eqref{dogra}). The restriction \eqref{eq:resr} on $r$ also appears in \cite{BBBM19} as a consequence of Condition 4.1 therein. Specializing to $K=\mathbb{Q}$, the above inequalities reduce to the equality $r=g$, and the condition \eqref{chK} then implies that
$g=r<g+\rho-1.$
This coincides with the condition assumed in the setting of the effective quadratic Chabauty method of \cite{balakrishnan2019,BDMTV21}.
\end{remark}

Even when $r\leq gd$, failures of the equality $d_J=r$ can occur. Examples include cases when the Jacobian is isogenous to a product of abelian varieties of large and small Mordell--Weil ranks.
For instance, suppose $C_\Q$ has genus $2$ and $J_\Q$ is isogenous to the product of two elliptic curves $E$ and $A$ with $\rank_{\Z} E(\Q)=0$ and $\rank_\Z A(\Q)=2$. Then $\dim \overline{E(\Q)}^p=0$ while $\dim \overline{A(\Q)}^p= 1$ since $A(\Q_p)$ has dimension $1$. Thus, $d_J= 1<r=g=2$. In the presence of such isogeny factors, the condition \eqref{chK} fails (conditional on the answer to Question \ref{qGuidoNrFields} being negative) and $\kappa$ is not finite-to-one. In particular, $\overline{A}_{\U,u}$ is infinite. Note that this issue is independent of the chosen prime and persists when varying $p$, providing a conditional (on a negative answer to Question \ref{qGuidoNrFields}) counter-example to the existence of a prime $p$ such that $\overline{A}_{\U,u}$ is finite, even over $\Q$. 

\begin{remark}
The existence of such conditional counter-examples to the hopes and expectations expressed by Edixhoven and Lido might be seen as an indication in favour of Question \ref{qGuidoNrFields} having an affirmative answer. It is worth noting though that the example treated by Edixhoven and Lido in \cite[\S 8]{EL19} satisfies $r=g=\rho=2$ with Jacobian isogenous to the product of two rank $1$ elliptic curves, and thus falls outside the small and large isogeny factor cases discussed above.
\end{remark}

\begin{remark}
For a related discussion of isogeny factors in the linear geometric Chabauty method over $\Q$, see \cite[Remark 2.1]{spelierhashimoto}. However, since $r<g$ is assumed therein, it is still possible to perform a linear method when $d_J<r$.  
\end{remark}

In conclusion, it seems to be slightly non-trivial to come up with the right theoretical condition for the effectiveness of the method. In practice, we do hope and expect that the method can be used to compute rational points in interesting new instances (as illustrated over $\Q$ in \cite[\S 8]{EL19}).




\addtocontents{toc}{\protect\setcounter{tocdepth}{1}}
\subsection*{Acknowledgements}

This project originated during the Arizona Winter School (AWS) in March 2020 and was proposed to us by Bas Edixhoven. We wish to thank the organizers of the conference for making this collaboration possible. We are grateful to Bas Edixhoven and Guido Lido for offering their insights and answering many questions regarding their paper. We further thank Bas Edixhoven for initial comments and suggestions on this article. We thank Jan Vonk and Guido Lido, who acted as the research project assistants for our group during the AWS, for their generous help. We also thank David Corwin, Netan Dogra, Deepam Patel, Pim Spelier, and Nicholas Triantafillou for helpful discussions. 
We thank the anonymous referee for their valuable suggestions which greatly helped improve the quality of the exposition, and for encouraging us to think about the geometric method in greater generality.

\subsection*{Funding}
During the preparation of this work, P\v{C} was partially supported by the Ross Fellowship, the Bilsland Fellowship, as well as Graduate School Summer Research Grants of Purdue University. DTBGL was partially supported by an Alexis and Charles Pelletier Fellowship and a Scholarship for Outstanding PhD Candidates from the Institut des Sciences Math\'ematiques (ISM) while at McGill University, and by an Emily Erskine Endowment Fund Postdoctoral Research Fellowship at the Hebrew University of Jerusalem. 
LXX was supported by the David and Barbara Groce travel fund, ERC Advanced grant 742608 ``GeoLocLang'', UMR 7586 IMJ-PRG and CNRS. ZY was partially funded by the European Research Council (ERC) under the European Union's Horizon 2020 research and innovation program (grant agreement No.~851146).

\phantomsection
\bibliographystyle{plain}
\bibliography{GQCNF_AM}

\end{document}